\documentclass[12p]{amsproc}

\usepackage{setspace}
\usepackage{latexsym}
\usepackage{amsmath, amssymb, amsfonts, amsthm}
\usepackage{float}
\usepackage{color}
\usepackage{mathrsfs}
\usepackage{hyperref} 
\hypersetup{citecolor=red, linkcolor=blue, colorlinks=true}  
\usepackage{algorithm}
\usepackage{enumerate}
\usepackage[noend]{algpseudocode}
\usepackage{chngcntr}

\newcounter{mycount}
\counterwithin{algorithm}{mycount}
\refstepcounter{mycount}

\usepackage{bm}
\usepackage[margin=1in]{geometry}
\usepackage{url}

\newtheorem{thm}{Theorem}[section]
\newtheorem{lem}[thm]{Lemma}
\newtheorem{prop}[thm]{Proposition}
\newtheorem{cor}[thm]{Corollary}

\newtheorem{conj}[thm]{Conjecture}

\newtheorem{defi}[thm]{Definition}

\newtheorem{rem}[thm]{Remark}

\renewcommand\le{\leqslant}
\renewcommand\ge{\geqslant}

\newcommand\Wr{{\rm \,wr\, }}

\newcommand\PSL{{\rm PSL}}
\newcommand\PGL{{\rm PGL}}
\newcommand\GAP{\textbf{{\rm GAP}}}
\newcommand\PSU{{\rm PSU}}
\newcommand\Sz{{\rm Sz}}
\newcommand\N{{\mathbb{N}}}
\newcommand\Aut{{\rm Aut}}
\newcommand\AGL{{\rm AGL}}
\newcommand\C{{\mathcal{C}}}
\newcommand\GL{{\rm GL}}
\newcommand\GF{{\rm GF}}
\newcommand\PGammaL{{\rm P}\Gamma{\rm L}}
\newcommand\ASL{{\rm ASL}}
\newcommand\PGammaU{{\rm P}\Gamma{\rm U}}
\newcommand\PSp{{\rm PSp}}
\newcommand\Sp{{\rm Sp}}
\newcommand\ASigmaL{{\rm A}\Sigma{\rm L}}
\newcommand\PSigmaL{{\rm P}\Sigma{\rm L}}
\newcommand\AGammaL{{\rm A}\Gamma{\rm L}}
\newcommand\SigmaU{{\Sigma{\rm U}}}
\newcommand\GO{{\rm GO}}
\newcommand\SU{{\rm SU}}
\newcommand\SL{{\rm SL}}
\newcommand\PGU{{\rm PGU}}
\newcommand\U{{\rm U}}
\newcommand\soc{{\rm soc}}
\newcommand\Out{{\rm Out}}

\newcommand\Sym{{\rm Sym}}

\newcommand{\calM}{\mathcal{M}}

\title{On integers that are covering numbers of groups}

\author{Martino Garonzi}
\address[Garonzi]{Departamento de Matem\'atica, Universidade de Bras\'ilia, Bras\'ilia, DF 70910-900, Brasil}
\email{mgaronzi@gmail.com}

\author{Luise-Charlotte Kappe}
\address[Kappe]{Department of Mathematical Sciences, State University of New York at Binghamton, Binghamton, NY 13902-6000, USA}
\email{menger@math.binghamton.edu}
\thanks{MG acknowledges the support of the Funda\c{c}\~{a}o de Apoio \`a Pesquisa do Distrito Federal (FAPDF) and the Coordena\c{c}\~{a}o de Aperfei\c{c}oamento de Pessoal de N\'ivel Superior - Brasil (CAPES)}

\author{Eric Swartz}
\address[Swartz]{Department of Mathematics, College of William \& Mary, P.O. Box 8795, Williamsburg, VA 23187-8795, USA}
\email{easwartz@wm.edu}

\begin{document}

 \begin{abstract}
The covering number of a group $G$, denoted by $\sigma(G)$, is the size of a minimal collection of proper subgroups of $G$ whose union is $G$. We investigate which integers are covering numbers of groups.  We determine which integers $129$ or smaller are covering numbers, and we determine precisely or bound the covering number of every primitive monolithic group with a degree of primitivity at most $129$ by introducing effective new computational techniques.  Furthermore, we prove that, if $\mathscr{F}_1$ is the family of finite groups $G$ such that all proper quotients of $G$ are solvable, then $\N-\{\sigma(G):G\in\mathscr{F}_1\}$ is infinite, which provides further evidence that infinitely many integers are not covering numbers.  Finally, we prove that every integer of the form $(q^m-1)/(q-1)$, where $m\neq3$ and $q$ is a prime power, is a covering number, generalizing a result of Cohn.
\end{abstract}

\keywords{Subgroup cover; Primitive group}
\subjclass[2010]{Primary 20D60; Secondary 20B15}

\maketitle

\section{Introduction}

A group $G$ is said to have a \textit{finite cover} by subgroups if it is the union of finitely many proper subgroups.  A cover of size $n$ of a group $G$ is called a \textit{minimal cover} if no cover of $G$ has fewer than $n$ subgroups.  Following J.~H.~E. Cohn \cite{Cohn}, we call the size of a minimal cover of a group $G$ the \textit{covering number}, denoted by $\sigma(G)$.  For a survey of results about the covering number of groups (and related results about analogously defined covering numbers of other algebraic structures), see \cite{KappeSurvey}.

The parameter $\sigma(G)$ has received a great deal of attention in recent years.  One reason for this is the connection to sets of pairwise generators of a group.  For a finite noncyclic group $G$ that can be generated by two elements, define $\omega(G)$ to be the largest integer $n$ such that there exists a set $S$ of size $n$ consisting of elements of $G$ such that every pair of distinct elements of $S$ generates $G$.  The covering number $\sigma(G)$ provides a natural and often tight upper bound for $\omega(G)$; see \cite{Blackburn, BEGHM1, BEGHM2, HolmesMaroti} for investigations on the relationship between these two parameters for various simple and almost simple groups.

It suffices to restrict our attention to finite groups when determining covering numbers, since, by a result of B.~H. Neumann \cite{Neumann}, a group is the union of finitely many proper subgroups if and only if it has a finite noncyclic homomorphic image.  Moreover, if a finite cover of a group $G$ exists, then we may realize a finite homomorphic image of $G$ by taking the quotient over the normal core of the intersection of the subgroups belonging to the finite cover.  Determining the covering number of a group $G$ predates Cohn's 1994 publication \cite{Cohn}.  It is easy to show that no group is the union of two proper subgroups.  Already in 1926, Scorza \cite{Scorza} characterized groups having covering number $3$ as those groups which have a homomorphic image isomorphic to the Klein four-group, a result forgotten and rediscovered later.

Cohn conjectures in \cite{Cohn} that the covering number of a noncyclic solvable group has the form $p^d + 1$, where $p$ is a prime and $d$ is a positive integer, and he shows that $\sigma(A_5) =10$ and $\sigma(S_5) = 16$.  In \cite{Tomkinson}, Tomkinson proves Cohn's conjecture and shows that there is no group with $\sigma(G) = 7$.  In addition, he conjectures that there are no groups with covering number $11$, $13$, or $15$.  Tomkinson's conjecture is confirmed only for the case $n = 11$ in \cite{DetomiLucchini}.  In fact, in \cite{AAS} it is shown that $\sigma(S_6) = 13$ and in \cite{BFS} that $\sigma(\PSL(2,7)) = 15$.  Furthermore, Tomkinson suggests that it might be of interest to investigate minimal covers of nonsolvable and, in particular, simple groups.  For an overview of the recent contributions addressing this question, we refer to \cite{KNS}.  The real question arising out of Tomkinson's results is to find all integers that are covering numbers and ascertain whether there are infinitely many integers that are not covering numbers.  The aim of this paper is to investigate which integers are covering numbers of groups.

To show that there are no groups $G$ for which $\sigma(G) = 7$, Tomkinson \cite{Tomkinson} proved that any group that can be covered by seven subgroups can actually be covered by fewer than seven subgroups.  In the long run, this is not the way to attack this problem.  In \cite{DetomiLucchini}, the authors observe that if a group $G$ exists with $\sigma(G) = n$, then there must exist a group $H$ with $\sigma(H) = n$ that has no homomorphic image with covering number $n$, making $H$ minimal in this sense.  The authors of \cite{DetomiLucchini} leverage this idea to show that no such group $G$ can exist with $\sigma(G) = 11$.  Originally introduced in \cite{DetomiLucchini} and following \cite{Garonzi3}, we say that a group $G$ is \textit{$\sigma$-elementary} if $\sigma(G) < \sigma(G/N)$ for every nontrivial normal subgroup $N$ of $G$.  For convenience, we say that the covering number of a cyclic group is infinite.  In \cite{Garonzi3}, the first author of this paper shows that a finite $\sigma$-elementary nonabelian group with covering number less than $26$ is either of affine type or an almost simple group with socle of prime index, which can then be used to determine which integers less than or equal to $25$ are covering numbers.  From earlier results and those in \cite{Garonzi3}, we have that the integers less than $26$ that are not covering numbers are $2$, $7$, $11$, $19$, $22$, and $25$.

In this paper, we extend this classification of integers that are covering numbers up to $129$.  We formulate this result by listing all integers between $26$ and $129$ that are not covering numbers.

\begin{thm}
\label{thm:129}
 The integers between $26$ and $129$ which are not covering numbers are $27$, $34$, $35$, $37$, $39$, $41$, $43$, $45$, $47$, $49$, $51$, $52$, $53$, $55$, $56$, $58$, $59$, $61$, $66$, $69$, $70$, $75$, $76$, $77$, $78$, $79$, $81$, $83$, $87$, $88$, $89$, $91$, $93$, $94$, $95$, $96$, $97$, $99$, $100$, $101$, $103$, $105$, $106$, $107$, $109$, $111$, $112$, $113$, $115$, $116$, $117$, $118$, $119$, $120$, $123$, $124$, $125$.
 \end{thm}

Theorem \ref{thm:129} follows from Theorem \ref{thm:129mono}, Proposition \ref{prop:solvable}, and Table \ref{tbl:sigmaelementary}.  To prove this result, we need to identify all potential candidates for $\sigma$-elementary groups with covering number between $26$ and $129$.  Toward this end, we show in Theorem \ref{thm:129mono} that the $\sigma$-elementary groups in question are among the groups with a unique minimal normal subgroup and degree of primitivity not exceeding $129$.  We say a finite group is \textit{primitive} if it admits a maximal subgroup $M$ with trivial normal core, and the index of $M$ in $G$ is called the primitivity degree of $G$ with respect to $M$.  A finite group that has a unique minimal normal subgroup is called \textit{monolithic}.

This characterization allows us to decide if a given integer in the range is a covering number or not.  Using $\GAP$ \cite{GAP}, we determine all nonsolvable monolithic groups with degree of primitivity between $26$ and $129$.  According to Theorem \ref{thm:129mono}, these are the candidates for nonabelian $\sigma$-elementary groups with covering number between $26$ and $129$ that are not solvable.  It is then necessary to determine -- or, at least, bound -- the covering number of these groups.  Either these results are known, or, if not, we have to attain them ourselves.  

Computationally, there are two main methods that we use.  The first is a method developed in \cite{KNS}, where a program in $\GAP$ creates a system of equations that can be solved (either partially or totally) by the linear optimization software GUROBI \cite{Gu}.  This method, which we refer to as Algorithm \hyperref[subsect:LP]{KNS}, is detailed in Section \ref{subsect:LP}.  The second method is introduced for the first time in this paper and is presented in Section \ref{subsect:verify}.  It is essentially a greedy algorithm, which works roughly as follows.  To build a cover, we take entire conjugacy classes of maximal subgroups.  Given a group $G$, we determine in $\GAP$ the minimum number of subgroups from a single conjugacy class of maximal subgroups needed to cover each conjugacy class of elements, and we choose the conjugacy class of elements that requires the maximum number of subgroups (among these minimums).  Rather than spending time checking precisely which maximal subgroups are absolutely necessary (something that the first method will do), all subgroups from the entire class of subgroups are chosen to be part of a cover.  All conjugacy classes of elements that are covered by subgroups in this class of maximal subgroups are removed, and this process is repeated until all elements of the group are covered.  Perhaps surprisingly, the cover produced this way is frequently a minimal cover and can be verified as such quickly using a simple calculation detailed in Lemma \ref{lem:keylemma}.  While the first method is extremely precise, it is often time and memory consuming, and it becomes impractical for groups of order more than half a million on a machine with a Core i7 processor and 16 GB of RAM.  The second method, while cruder, runs much faster in practice, and can essentially always be used to provide both upper and lower bounds for the covering number when $\GAP$ can determine the maximal subgroups and conjugacy classes of a given group.  Pseudocode for this method is provided in Algorithm \ref{alg}.  When neither of these methods is totally effective, ad hoc methods are used; for this, see Appendix \ref{sect:calcspecific}.

It is a natural question to ask if every nonabelian $\sigma$-elementary group is a monolithic primitive group.  In this respect, the authors of \cite{DetomiLucchini} make the following conjecture.

\begin{conj}\cite{DetomiLucchini}  
\label{conj:DL}
Every nonabelian $\sigma$-elementary group is a monolithic primitive group.
\end{conj}

So far, no counterexamples to Conjecture \ref{conj:DL} are known.  In Remark \ref{rem:DL}, it is observed that if $\sigma(X) < 2\sigma^{\ast}(X)$ (see Definition \ref{defstar}) for all primitive monolithic groups $X$ with a nonabelian socle, then Conjecture \ref{conj:DL} is true.  (See also the preceding Lemma \ref{2n+1}, which represents how close we are to proving Conjecture \ref{conj:DL}.)  For the proof of Theorem \ref{thm:129}, we establish in the proof of Theorem \ref{thm:129mono} that this inequality holds at least when $\sigma(X) < 130$.  The question arises if this inequality can be extended to bounds larger than $130$.  On the other hand, the techniques used to determine the covering numbers of the candidate groups with primitivity degree at most $129$ were often pushed to the limit.  (For instance, there is one group in particular whose covering number cannot be determined to any smaller range than between $138$ and $166$ using current methods; for this, see Table \ref{tbl:main5}.)  Extending the bound of the primitivity degree may then necessitate new methods (or the use of extremely powerful computers) for determining covering numbers.

After the conjectures of Tomkinson \cite{Tomkinson} were settled, i.e., precisely which integers up to $18$ are not covering numbers, only three out of the $17$ integers between $2$ and $18$ were found not to be covering numbers, around $18 \%$.  One possibility is that there are only finitely many integers that are not covering numbers.  The statistics following Theorem \ref{thm:129} tell us that around fifty percent of the integers less than $130$ are not covering numbers, leading us to make the following conjecture with confidence.

\begin{conj}
\label{conj:infinite}
Let $\mathscr{E}$ be the set of all integers that are covering numbers.  Then, the set $\mathbb{N}-\mathscr{E}$ is infinite, in other words there are infinitely many natural numbers that are not covering numbers.
\end{conj}

Our results here about integers that are not covering numbers in a certain range were obtained by determing the complement in this range.  Indications are that a proof of Conjecture \ref{conj:infinite} requires a similar approach.  By \cite[Theorem 1]{DetomiLucchini}, a $\sigma$-elementary group $G$ with no abelian minimal normal subgroups is of one of two types:
\begin{itemize}
\item[(1)] $G$ is a primitive monolithic group such that $G/\soc(G)$ is cyclic, or 
\item[(2)] $G/\soc(G)$ is nonsolvable, and all the nonabelian composition factors of $G/\soc(G)$ are alternating groups of odd degree.  
\end{itemize}
The following result represents progress toward proving Conjecture \ref{conj:infinite} by dealing with the groups of type (1).

\begin{thm} \label{densityth}
Let $\mathcal{G}$ be the family of primitive monolithic groups $G$ with nonabelian socle such that $G/\soc(G)$ is cyclic. Then there exists a constant $c$ such that for every $x>0$, $$|\{\sigma(G) : G \in \mathcal{G},\ \sigma(G) \le x\}| \le c x^{5/6}.$$In particular, $\mathbb{N}-\{\sigma(G) : G \in \mathcal{G}\}$ is infinite.
\end{thm}

Given Tomkinson's result about the covering number of solvable groups and in light of Conjecture \ref{conj:DL} and Theorem \ref{densityth}, it is perhaps probable even that almost all integers are not covering numbers; that is, we make the following conjecture.

\begin{conj}
 \label{conj:infdens}
Let $\mathscr{E}(n) := \{ m : m \le n, \sigma(G) = m \text{ for some group $G$}\}.$  Then,  \[ \lim_{n \to \infty} \frac{|\mathscr{E}(n)|}{n} = 0.\]
\end{conj}

As a corollary of Theorem \ref{densityth}, we prove the following result, which can be viewed as additional evidence for the validity of Conjecture \ref{conj:infinite}.

\begin{cor}
\label{cor:density}
Let $\mathscr{F}_1$ be the family of finite groups $G$ such that all proper quotients of $G$ are solvable. Then the set $\mathbb{N}-\{\sigma(G) : G \in \mathscr{F}_1\}$ is infinite.
\end{cor}

Indeed, one may view Corollary \ref{cor:density} as a generalization of Tomkinson's result about the covering numbers of solvable groups (see Proposition \ref{prop:solvable} (ii)) in the following sense.  Let $\mathscr{F}_n$ denote the set of all finite groups with at most $n$ nonsolvable quotients. Note that $\mathscr{F}_i \subset \mathscr{F}_j$ for all $i<j$, and, if we define $\mathscr{F} := \bigcup_{n = 0}^\infty \mathscr{F}_n$, then $\mathscr{F}$ is the family of all finite groups.  If for a collection of groups $\mathscr{C}$ we define 
\[\mathscr{E}(\mathscr{C}) := \{\sigma(G): G \in \mathscr{C}\},\]
then Conjecture \ref{conj:infinite} is that $\N - \mathscr{E}(\mathscr{F})$ is infinite.  A consequence of Tomkinson's result is that $\N - \mathscr{E}(\mathscr{F}_0)$ is infinite, and Corollary \ref{cor:density} above states that $\N - \mathscr{E}(\mathscr{F}_1)$ is infinite.  The next step toward proving Conjecture \ref{conj:infinite} would be to prove that $\N - \mathscr{E}(\mathscr{F}_2)$ is infinite, and for this one needs to determine (among other things) the covering number of affine groups, i.e., primitive groups with an abelian socle.  As a first step in this process, we prove the following.

\begin{thm}
 \label{thm:AGL}
Let $q = p^d$, where $p$ is prime and $d \in \N$, and let $n \ge 1$, $n \neq 2$ be a positive integer.  Then \[ \sigma(\AGL(n,q)) = \sigma(\ASL(n,q)) = \frac{q^{n+1} - 1}{q-1}.\] In particular, for all $m \ge 2$, $m \neq 3$, $(q^m - 1)/(q-1)$ is a covering number. 
\end{thm} 

This result can also be viewed as a generalization of a result of Cohn \cite[Corollary to Lemma 17]{Cohn}. There it is shown that all integers of the form $(q^2 - 1)/(q-1) = q + 1$, where $q$ is a prime power, are covering numbers.

Our hope is to obtain further density results along the lines of Theorem \ref{densityth} for $\sigma$-elementary groups and eventually arrive at a proof of Conjecture \ref{conj:infinite}.  Other than the remaining primitive groups with abelian socle, we would like to highlight the family of wreath products $S \Wr K$, where $S$ is a nonabelian simple group and $K$ is a transitive group of degree $k$.  These groups are an archetypal case among those that remain, and it would be interesting to prove results for this family in particular.  

The structure of this paper is as follows. Section \ref{sect:prelim} contains preparatory results for subsequent sections, especially the following two sections.  Section \ref{sect:density} contains the proof of Theorem \ref{densityth}, the density result for a certain class of $\sigma$-elementary groups, and Section \ref{sect:129} contains the proof of the necessary condition that a nonabelian $\sigma$-elementary group $G$ with $\sigma(G) \le 129$ is a monolithic primitive group with a degree of primitivity of at most $129$.  Section \ref{sect:comp} details the two main computational methods used to determine covering numbers.

The covering numbers or estimates of them are known for some classes of monolithic primitive groups that are candidates to be $\sigma$-elementary groups with covering number at most $129$, and these results are summarized in Section \ref{sect:known}.  The covering number for groups in two other families of monolithic primitive groups, the affine general linear groups and the affine special linear groups, are determined in Section \ref{sect:affine}, which also shows that every integer of the form $(q^n - 1)/(q-1)$, where $q$ is a prime power and $n > 3$, is a covering number.  All our results for calculations and bounds are summarized in tables in Section \ref{sect:tables}, among which we have a list of the nonsolvable $\sigma$-elementary groups with $\sigma(G) \le 129$ in Table \ref{tbl:sigmaelementary}.  If $p$ is a prime and $d$ is a positive integer, then this table together with the integers in this range of the form $p^d + 1$ (that is, the covering numbers of solvable groups by \cite{Tomkinson}) establish Theorem \ref{thm:129}.  Finally, Appendix \ref{sect:calcspecific} contains calculations or bounds for the covering number for the various groups that could not be dealt with using the methods of the previous sections.

\section{Background}
\label{sect:prelim}

In this section for the convenience of the reader we begin with some well-known concepts used throughout this paper and present some relevant definitions and lemmas which can be found in earlier publications addressing this topic.  The uninitiated reader is encouraged to consult \cite{DixonMortimer} or \cite{Robinson}. 

The \textit{socle} of a finite group $G$, denoted $\soc(G)$, is the subgroup of $G$ generated by the minimal normal subgroups of $G$, and, in fact, the socle of $G$ is a direct product of some minimal normal subgroups of $G$. A finite group $G$ is said to be \textit{monolithic} if it admits a unique minimal normal subgroup, which therefore equals the socle of $G$.  

A finite group $G$ is said to be \textit{primitive} if it admits a maximal subgroup $M$ such that $M_G = \bigcap_{g \in G} g^{-1}Mg$ (the \textit{normal core} of $M$) is trivial, and in this case the index $|G:M|$ is called a \textit{primitivity degree} (or \textit{degree of primitivity}) of $G$; a primitive group in general has many primitivity degrees.  This definition for abstract groups is equivalent to the permutation group definition; that is, if $\Omega$ is a finite set, then $G \le \Sym(\Omega)$ is primitive on the set $\Omega$ if and only if $G$ is transitive on and stabilizes no nontrivial partition of $\Omega$.  (The equivalence can be seen by taking the set of right cosets of $M$ in $G$ to be $\Omega$ under the action of right multiplication.)  It is well-known that a finite primitive group $G$ is either monolithic or it admits precisely two minimal normal subgroups, and, in this case, such minimal normal subgroups are nonabelian and isomorphic; see, for instance, \cite[Theorem 1.1.7]{spagn}. 

Any minimal normal subgroup of a finite group $G$ is \textit{characteristically simple}, so it has the form $S^r$ for some simple group $S$ (which could be abelian) and some positive integer $r$. A minimal normal subgroup $N$ of $G$ is said to be \textit{supplemented} if it admits a \textit{supplement} in $G$, that is, a proper subgroup $H$ of $G$ such that $HN=G$, and \textit{complemented} if it admits a \textit{complement} in $G$, that is, a supplement $H$ such that $H \cap N = \{1\}$. The minimal normal subgroup $N$ is said to be \textit{Frattini} if it is contained in the Frattini subgroup of $G$, which is the intersection of the maximal subgroups of $G$, and \textit{non-Frattini}, otherwise. (Note that a Frattini minimal normal subgroup is automatically abelian, since the Frattini subgroup is nilpotent.)  

We now present some results regarding the Frattini subgroup and minimal normal subgroups.  For a minimal normal subgroup $N$ of $G$, being non-Frattini is equivalent to being supplemented, where the supplement of $N$ is any maximal subgroup of $G$ which does not contain $N$.  For an abelian minimal normal subgroup $N$ of $G$, being supplemented is equivalent to being complemented: if $H$ is a supplement of $N$ in $G$, then $H \cap N \neq N$, since $HN=G$; $H \cap N$ is normal in $H$, since $N$ normal in $G$; and $H \cap N$ is normal in $N$, since $N$ is abelian. Hence, we have $N \cap H \unlhd NH = G$, and $N \cap H = \{1\}$, since $N$ a minimal normal subgroup. Also, observe that a monolithic group $G$ is primitive if and only if the Frattini subgroup of $G$ is trivial: indeed, if $G$ is primitive, then it is clear from the definition that the Frattini subgroup of $G$ is trivial; conversely, suppose that $G$ is monolithic with trivial Frattini subgroup.  Since $G$ is monolithic, the socle of $G$ is contained in every nontrivial normal subgroup of $G$, and, since $G$ has trivial Frattini subgroup, there must exist a maximal subgroup with trivial normal core, since, otherwise, the Frattini subgroup would contain the socle.  Hence, $G$ is primitive. 

Following \cite{DetomiLucchini2}, the \textit{primitive monolithic group} $X_N$ associated to a non-Frattini minimal normal subgroup $N$ of a group $G$ is defined as follows, based on whether or not $N$ is abelian:

\begin{itemize}
\item If $N$ is abelian, then there exists a complement $H$ of $N$ in $G$. Then $C_H(N) \unlhd G$ and we define $X_N:=G/C_H(N)$.  
\item If $N$ is nonabelian, then we define $X_N:=G/C_G(N)$.
\end{itemize}

The following observations can be proved easily.  In the case when is $N$ abelian the primitive monolithic group associated to $N$ depends on the choice of complement; however, choosing a different complement gives an isomorphic primitive group.  In any case, $X_N$ is a primitive monolithic group with socle isomorphic with $N$ (the socle is $NC_H(N)/C_H(N)$ in the first case and $NC_G(N)/C_G(N)$ in the second case).  Observe that if $G$ is itself primitive and monolithic, then $G$ coincides with the primitive monolithic group associated with its socle: if $G$ is primitive and monolithic with socle $N$, then the centralizer of $N$ in $G$ is trivial if $N$ is nonabelian and equals $N$ if $N$ is abelian; otherwise, a larger centralizer would give rise to a nontrivial normal subgroup not containing $N$ in both cases.  

Let $\sigma(G)$ denote the covering number of $G$, with $\sigma(G) = \infty$ if $G$ is cyclic (with the convention that $n < \infty$ for all integers $n$). It is easy to see that $\sigma(G) \le \sigma(G/N)$ for all normal subgroups $N$ of $G$, and, with this in mind, we have the following definition.

\begin{defi}
A finite noncyclic group $G$ is said to be \textit{$\sigma$-elementary} if $\sigma(G) < \sigma(G/N)$ for all nontrivial normal subgroups $N$ of $G$. 
\end{defi}

In the following result, $\Phi(G)$ denotes the Frattini subgroup of $G$.

\begin{lem}[\cite{DetomiLucchini} Corollary 14] \label{sigmaprim}
Let $G$ be a finite $\sigma$-elementary group. If $G$ is abelian, then $G \cong C_p \times C_p$ for some prime $p$. If $G$ is nonabelian, then the following hold:
\begin{enumerate}
\item[(1)] The Frattini subgroup $\Phi(G)$ of $G$ is trivial.
\item[(2)] $G$ has at most one abelian minimal normal subgroup;
\item[(3)] Let $\soc(G)=N_1 \times \cdots \times N_n$ be the socle of $G$, where $N_1,\dots,N_n$ are the minimal normal subgroups of $G$. Then $G$ is a subdirect product of the primitive monolithic groups $X_i$ associated to the $N_i$'s, and the natural map from $G$ to $X_1 \times \cdots \times X_n$ given by the natural projections of $G$ onto each $X_i$ is injective.
\end{enumerate}
\end{lem}

We also note the following additional structural result.

\begin{lem}[\cite{Cohn} Theorem 4]
 \label{lem:trivcenter}
 If $G$ is a nonabelian $\sigma$-elementary group, then the center of $G$ is trivial.
\end{lem}

The following definition provides a concept that is useful for bounding the covering number of a group from below.

\begin{defi}[\cite{DetomiLucchini} Definition 15] \label{defstar}
Let $X$ be a primitive monolithic group with socle $N$. If $\Omega$ is an arbitrary union of cosets of $N$ in $X$, define $\sigma_{\Omega}(X)$ to be the smallest number of supplements of $N$ in $X$ needed to cover $\Omega$. Define $$\sigma^{\ast}(X) := \min \{\sigma_{\Omega}(X)\ |\ \Omega = \bigcup_i \omega_i N,\ \langle \Omega \rangle = X \}.$$
\end{defi}

The values of $\sigma^{\ast}(X_i)$, where the $X_i$ are as in Lemma \ref{sigmaprim}, provide a lower bound for $\sigma(G)$ when $G$ is a $\sigma$-elementary group in terms of the primitive monolithic groups associated to its minimal normal subgroups. The following is a useful lower bound.

\begin{lem}[\cite{DetomiLucchini} Proposition 16] \label{sigmastar}
Let $G$ be a nonabelian group, let $\soc(G) = N_1 \times \ldots \times N_n$, and let $X_1,\ldots,X_n$ be the primitive monolithic groups associated to $N_1, \ldots, N_n$, respectively. If $G$ is $\sigma$-elementary, then $$\sigma^{\ast}(X_1) + \ldots + \sigma^{\ast}(X_n) \le \sigma(G).$$
\end{lem}

For a primitive monolithic group $X$ with socle $N$, we denote by $\ell_X(N)$ the minimal index of a proper supplement of $N$ in $X$.  In other words, $\ell_X(N)$ is the smallest primitivity degree of $X$.

\begin{lem}[\cite{DetomiLucchini} Remark 17] \label{indexbound}
If $X$ is a primitive monolithic group, then $$\sigma^{\ast}(X) \ge \ell_X(\soc(X)).$$If $G$ is a nonabelian $\sigma$-elementary group, $N_1,\ldots,N_n$ are the minimal normal subgroups of $G$, and $X_1,\ldots,X_n$ are the primitive monolithic groups associated to $N_1,\ldots,N_n$, respectively, then $$\sum_{i=1}^n \ell_{X_i}(N_i) \le \sum_{i=1}^n \sigma^{\ast}(X_i) \le \sigma(G).$$ In particular, for every $i \in \{1,\ldots,n\}$, the group $X_i$ has primitivity degree at most $\sigma(G)$.
\end{lem}

The following lemmas from \cite{DetomiLucchini} are critical in later proofs. Note that Lemma \ref{abmns} below includes information contained in the proof as well as the statement of \cite[Proposition 10]{DetomiLucchini}.

\begin{lem}[\cite{DetomiLucchini} Lemma 18] \label{spancoprime}
Let $N$ be a normal subgroup of a group $X$. If a set of subgroups of $X$ covers a coset $yN$ of $N$ in $X$, then it also covers every coset $y^{\alpha}N$ with $\alpha$ prime to $|y|$.
\end{lem}

\begin{lem} [\cite{DetomiLucchini} Proposition 21] \label{solcyc}
Let $G$ be a nonabelian $\sigma$-elementary group. If a proper quotient $G/N$ is solvable, then it is cyclic.
\end{lem}

\begin{lem} [\cite{DetomiLucchini} Proposition 10] \label{abmns}
Let $G$ be a group. If $V$ is a complemented normal abelian subgroup of $G$ and $V \cap Z(G) = \{1\}$, then $\sigma(G) \le 2|V|-1$.  In particular, if $V$ is a minimal normal subgroup, where $q = |{\rm End}_G(V)|$ and $|V| = q^n$, and $H$ is a complement of $V$ in $G$, then the collection
\[ \{H^v : v \in V\} \cup \{C_H(W)V : W \le V, \dim_{\GF(q)}(W) = 1\}\]
is a cover for $G$ and
\[\sigma(G) \le 1 + q + \cdots + q^n = \frac{q^{n+1}-1}{q-1}.\]
\end{lem}

Finally, the following lemmas prove to be extremely useful when calculating covering numbers.  The first lemma is a straightforward criterion for showing that a maximal subgroup is contained in any minimal cover containing only maximal subgroups.

\begin{lem}\cite[Lemma 1]{Garonzi3}
\label{lem:Garonzi}
 If $H$ is a maximal subgroup of a group $G$ and $\sigma(H) > \sigma(G)$, then $H$ appears in every minimal cover of $G$ containing only maximal subgroups.  In particular, if $H$ is maximal and non-normal then $\sigma(H) < [G:H]$ implies $\sigma(G) \ge \sigma(H)$.
\end{lem}

The final lemma of this section is due to Detomi and Lucchini (also proved independently by S. M. Jafarian Amiri \cite{Amiri}) and is useful when proving results about covering numbers of primitive groups with an elementary abelian minimal normal subgroup.

\begin{lem}\cite[Corollary 6]{DetomiLucchini}
\label{lem:Amiri}
If $G$ is a primitive group with stabilizer $H$ and unique abelian minimal normal subgroup $N$, then $\sigma(G) \ge |N| + 1$ or $\sigma(G) = \sigma(H)$.
\end{lem}

\section{A density result} 
\label{sect:density}

One of the main problems about group coverings is the following: what does the set of numbers of the form $\sigma(G)$, where $G$ is a finite group, look like? 

Recall that Conjecture \ref{conj:infinite} hypothesizes that there are infinitely many natural numbers that are not covering numbers.  A good strategy to approach this conjecture is the following: first, find a specific (and ``easy'' to handle) family $\mathscr{F}$ of groups such that $\{\sigma(G) : G \in \mathscr{F}\} = \mathscr{E}$ and then deal with the family $\mathscr{F}$. If Conjecture \ref{conj:DL} is true, then, because we clearly can choose as $\mathscr{F}$ the family of $\sigma$-elementary groups, we may choose as $\mathscr{F}$ the family of primitive monolithic groups. An important subfamily of it is the family of primitive monolithic groups whose quotient over the socle is cyclic, since the proper solvable quotients of $\sigma$-elementary groups are cyclic. In this section, we show that the density of the values $\sigma(G)$ for $G$ a primitive monolithic group with $G/\soc(G)$ cyclic is zero; specifically, setting $\mathcal{G}$ to be the family of such primitive monolithic groups, we show that $$|\{\sigma(G) : G \in \mathcal{G},\ \sigma(G) \le x\}| \le c x^{5/6}$$ for some constant $c$. This implies that $\mathbb{N}-\{\sigma(G) : G \in \mathcal{G}\}$ is infinite.

Before we can proceed with the proof of the main theorem in this section, we need some preparatory results.  The first can be considered part of the O'Nan-Scott Theorem (see \cite[Remark 1.1.40]{spagn}). Let $G$ be a primitive monolithic group with nonabelian socle $N=T^m$. Let $H$ be a maximal subgroup of $G$ such that $N \not \subseteq H$, i.e. $HN=G$ and $H$ supplements $N$. Suppose $H \cap N \neq \{1\}$, i.e., $H$ does not complement $N$. Since $N$ is a minimal normal subgroup of $G$ and $H$ is a maximal subgroup of $G$ not containing $N$, $H = N_G(H \cap N)$. In the following, let $X := N_G(T_1)/C_G(T_1)$, which is an almost simple group with socle $T_1C_G(T_1)/C_G(T_1) \cong T$. There are two possibilities for the intersection $H \cap N$, and the primitive group $G$ is described as having one of two types, depending on the possibility.  These two types and some basic properties of each are described as follows. 

\begin{enumerate}
\item \textbf{Product type}. In this case, the projections $H \cap N \to T_i$ are not surjective. This implies that there exists a subgroup $M$ of $T$, which is an intersection of $T$ with a maximal subgroup of $X$, such that $N_X(M)$ supplements $T$ in $X$, and there exist elements $a_2,\ldots,a_m \in T$ such that $H \cap N$ equals $$M \times M^{a_2} \times \ldots \times M^{a_m}.$$

\item \textbf{Diagonal type}. In this case, the projections $H \cap N \to T_i$ are surjective. This implies that there exists a minimal $H$-invariant partition $P$ of $\{1,\ldots,m\}$ into imprimitivity blocks of the action of $H$ on $\{1,\ldots,m\}$ such that $H \cap N$ equals $$\prod_{D \in P} (H \cap N)^{\pi_D},$$and, for each $D \in P$, the projection $(H \cap N)^{\pi_D}$ is a \textit{full diagonal} subgroup of $\prod_{i \in D}T_i$. (Following \cite[Definition 1.1.37]{spagn}, a subgroup $H$ of $\prod_{i \in D}T_i$ is said to be \textit{full diagonal} if each projection $\pi_i:H \to T_i$ is an isomorphism.)
\end{enumerate}

For a finite nonabelian simple group $T$, denote by $m(T)$ the minimal index of a proper subgroup of $G$, which is equal to the minimal degree of a transitive permutation representation of $T$.  Recall that $\ell_G(N)$ denotes the minimal index of a proper supplement of $N$ in $G$.  The following lemma provides information about $\ell_G(N)$ for primitive monolithic groups $G$ with socle $N$.

\begin{lem} \label{five}
Let $G$ be a primitive monolithic group with socle $N$. If $N$ is abelian, then $\ell_G(N)=|N|$. If $N$ is nonabelian, then write $N=T^r$ with $T$ a nonabelian simple group. Let $H$ be a maximal subgroup of $G$ supplementing $N$.
\begin{itemize}
\item[(i)] If $H$ complements $N$, then $|G:H| = |N| = |T|^r$.
\item[(ii)] If $H$ has product type, then $H = N_G(M \times M^{a_2} \times \ldots \times M^{a_r})$ for some subgroup $M$ of $T$ of the form $Y \cap T$, where $Y$ is a maximal subgroup of $X$ supplementing $T$, and $|G:H| = |T:M|^r$.
\item[(iii)] If $H$ has diagonal type, then $H = N_G(\Delta)$, where $\Delta$ is a product of $r/c$ diagonal subgroups (in the sense of the above description of diagonal type) with $c$ a prime divisor of $r$ larger than $1$, and $|G:H| = |T|^{r-r/c}$.
\end{itemize}
Moreover $\ell_G(N) \ge m(T)^r$.
\end{lem}

\begin{proof}
Suppose $N$ is abelian. Since $G$ is primitive, $N$ is non-Frattini, so it is complemented and each of its complements have index $\ell_G(N) = |N|$.

Suppose $N$ is nonabelian. The three listed facts in the statement follow easily from the fact that $|G:H| = |N:H \cap N|$. Now let us prove that $\ell_G(N) \ge m(T)^r$. Since $m(T)^r \le |T:M|^r$ for every proper subgroup $M$ of $T$, it suffices to show that $m(T)^r \le |T|^{r-r/c}$ for every divisor $c > 1$ of $r$. For this, it is enough to show that $m(T)^r \le |T|^{r/2}$, i.e., $m(T)^2 \le |T|$. This is true by inspection, using \cite{lucdam}.
\end{proof}

Next, we will prove a technical lemma which basically says under the right conditions that, if a set of numbers is ``small,'' then the set of all possible powers of those numbers is also small.

\begin{lem}
\label{powers}
Let $A$ be a subset of $\mathbb{N}$, and for $x \in \mathbb{R}$ let $$\theta(x) := |\{n \in A\ :\ n \le x\}|.$$ If there exists a constant $c$ such that $\log(x) \theta(x^{\frac{1}{2}}) \le c \theta(x)$ for every $x > 0$, then there exists a constant $d$ such that $$|\{n^k\ :\ n \in A,\ k \in \mathbb{N},\ n^k \le x\}| \le d \theta(x)$$for every $x > 0$.
\end{lem}

\begin{proof}
Let $N(x)$ be the smallest natural number such that $2^{N(x)} > x$. Clearly there exists a constant $b$ such that $N(x) \le b \log(x)$, and 
\begin{eqnarray}
|\{n^k\ :\ n \in A,\ k \in \mathbb{N},\ n^k \le x\}| & \le & \theta(x) + \theta(x^{1/2}) + \ldots + \theta(x^{1/N(x)}) \nonumber \\ & \le & \theta(x) + N(x) \theta(x^{1/2}) \le d \theta(x), \nonumber
\end{eqnarray}
where $d = 1+bc$.
\end{proof}

The following result is due to Frobenius; see \cite[Section 5]{embthm} for a modern treatment. Here, if $H$ and $K$ are two groups and $K \le S_n$, then $H \Wr K$ denotes the wreath product between $H$ and $K$, i.e., the semidirect product $H^n \rtimes K$, where $K$ acts on $H^n$ by permuting the coordinates.

\begin{thm}\label{embedding}
Let $H$ be a subgroup of the finite group $G$, let $x_1,\ldots,x_n$ be a right transversal for $H$ in $G$, and let $\xi$ be any homomorphism with domain $H$. Then the map $G \to \xi(H) \Wr S_n$ given by $$x \mapsto (\xi(x_1 x x_{1^{\pi}}^{-1}), \ldots, \xi(x_n x x_{n^{\pi}}^{-1})) \pi,$$where $\pi \in S_n$ satisfies $x_i x x_{i^{\pi}}^{-1} \in H$ for all $i = 1,\ldots,n$, is a well-defined homomorphism with kernel equal to the normal core $(\ker \xi)_G$.
\end{thm}

Next, in the following remark we establish some notation and basic results about monolithic groups with a nonabelian socle.

\begin{rem} \label{mono}
Let $G$ be a monolithic group with socle $N = \soc(G) = T_1 \times \cdots \times T_m$, where $T_1, \ldots ,T_m$ are pairwise isomorphic nonabelian simple groups. We also define $X := N_G(T_1)/C_G(T_1)$, which is an almost-simple group with socle $T := T_1 C_G(T_1)/C_G(T_1) \cong T_1$. The minimal normal subgroups of $T^m = T_1 \times \ldots \times T_m$ are precisely its factors $T_1,\ldots,T_m$. Since automorphisms send minimal normal subgroups to minimal normal subgroups, it follows that $G$ acts on the $m$ factors of $N$. Let $\rho: G \to S_m$ be the homomorphism induced by the conjugation action of $G$ on the set $\{T_1, \ldots,T_m\}$. The group $K := \rho(G)$ is a transitive permutation group of degree $m$.  Choosing $H := N_G(T_1)$ and $\xi: H \to \mbox{Aut}(T_1)$, the homomorphism given by the conjugation action of $H$ on $T_1$, by Theorem \ref{embedding}, we see that $G$ embeds in the wreath product $X \Wr K$.
\end{rem}

We need some consequences of the classification of finite simple groups (henceforth CFSG); see \cite{GLS}. For $T$ a finite nonabelian simple group, recall that $m(T)$ denotes the minimal index of a proper subgroup of $T$. Clearly $m(A_n)=n$, and the value of $m(T)$ when $T$ is a group of Lie type can be found in \cite[Table 1]{lucdam}.

\begin{lem} \label{simple}
Let $T$ be a nonabelian finite simple group.
\begin{itemize}
\item[(1)] There exists a constant $c$ such that $|\Out(T)| \le c \log(m(T))$.
\item[(2)] If $T$ is non-alternating and not of the form $\PSL(2,q)$, then there are at most $cx^{\frac{1}{2}}/\log(x)$ groups $T$ such that $m(T) \le x$, where $c$ is a constant.
\end{itemize}
\end{lem}

\begin{proof}
Item (1) follows from CFSG by inspection. For (2), by CFSG, if $q$ is the size of the base field and $T$ is not $\PSL(2,q)$, then there always is a constant $b$ such that $bq^2 \le m(T)$ so that $m(T) \le x$ implies $q \le (x/b)^\frac{1}{2}$, and by the Prime Number Theorem and Lemma \ref{powers}, there are at most $c x^\frac{1}{2}/\log(x)$ choices for $q$, where $c$ is a constant.  For a given $q$, we need only consider the constant number of families where $m(T) \sim q^2$. Indeed, if $bq^3 \le m(T)$, then we may replace the square root by a cube root, and there are on the order of $\log(x)$ possible values of $n$, and so there are at most $\left(d x^\frac{1}{3}/\log(x)\right) \cdot \log(x)$ total possibilities outside the families when $m(T) \sim q^2$, where $d$ is a constant.  The result follows.
\end{proof}

Let $\mathcal{G}$ be a family of monolithic $\sigma$-elementary groups with nonabelian socle, and for $G \in \mathcal{G}$, let $\soc(G) = T^k$ for $T$ a nonabelian simple group and let $n_{\sigma}(G) := m(T)^k$.  The following lemma provides bounds for the number of integers that are covering numbers of groups in $\mathcal{G}$ in terms of the number of integers that are of the form $n_\sigma(G)$ for $G \in \mathcal{G}$.

\begin{lem} \label{g(x)}
Let $\mathcal{H}$ be a subfamily of $\mathcal{G}$. Define $$A := \{\sigma(G) : G \in \mathcal{H}\}, \;\;  B := \{n_{\sigma}(G) : G \in \mathcal{H}\}.$$Let $g(x)$ be a function such that, for all $n \le x$, $$|\{G \in \mathcal{H} : n_{\sigma}(G)=n\}| \le g(x).$$ Then $$|\{n \in A : n \le x\}| \le g(x) \cdot |\{n \in B : n \le x\}|.$$
\end{lem}

\begin{proof}
Indeed,
\begin{align*}
|\{n \in A : n \le x\}| & \le |\{G \in \mathcal{H} : \sigma(G) \le x\}| \le |\{G \in \mathcal{H} : n_{\sigma}(G) \le x\}| \\ & = \sum_{n \le x} |\{G \in \mathcal{H} : n_{\sigma}(G)=n\}| \le g(x) \cdot |\{n \in B : n \le x\}|.
\end{align*}
\end{proof}

Next, we restate a lemma from \cite{Tomkinson}, since there is a misprint in the original version.

\begin{lem}
 \label{lem:tom3.2}
 Let $N$ be a proper subgroup of the finite group $G$.  Let $U_1, \dots, U_k$ be proper subgroups of $G$ containing $N$ and $V_1, \dots, V_k$ be subgroups such that that $V_iN = G$ with $|G:V_i| = \beta_i$ and $\beta_1 \le \dots \le \beta_k$.  If 
 \[G = U_1 \cup \dots \cup U_h \cup V_1 \cup \dots \cup V_k,\]
 where $U_1 \cup \dots \cup U_h \neq G$, then $\beta_1 \le k$.  Furthermore, if $\beta_1 = k$, then $\beta_1 = \dots = \beta_k = k$ and $V_i \cap V_j \le U_1 \cup \dots \cup U_h$ for all $i \neq j$.
\end{lem}

We are now ready to prove Theorem \ref{densityth} and Corollary \ref{cor:density}.

\begin{proof}[Proof of Theorem \ref{densityth}]
Let us use the notation established in Remark \ref{mono}. Let $g$ be an element of $G$ which generates $G$ modulo $\soc(G)$. We know that $G$ embeds in the wreath product $X \Wr K$, so $g$ has the form $(x_1,\ldots,x_k)\tau$, where $x_1,\ldots,x_k \in X$ and $\tau \in K$ is a $k$-cycle in $S_k$ that generates $K$. Moreover, without loss of generality, we may assume that $\tau$ is the $k$-cycle $(1 \; 2 \; \dots \; k)$.  Observe that conjugating $\tau$ by $(1,1,\ldots,1,x)$ for any $x \in X$ gives $(1,1,\ldots,1,x,x^{-1}) \tau$, and conjugating elements of $T^k$ by $\tau$ has the effect of ``cycling'' the coordinates. This implies that up to replacing $G$ by a conjugate of $G$ in $X \Wr K$ we may assume that $g = (1,\ldots,1,x)\tau$, where $x$ is a generator of $X$ modulo $T$; such a generator $x$ must exist since $G/\soc(G)$ is cyclic. Since $X/T \le \Out(T)$, this implies that in $\mathcal{G}$ there are at most $|\Out(T)|$ isomorphism classes of groups $G$ with given socle $T^k$. 

We claim that for fixed $j \le x$ the number of simple groups $T$ with $m(T)=j$ is at most $cx^\frac{1}{3}$, where $c$ is a constant. This is transparent in the case of sporadic and alternating groups. Groups of Lie type are parametrized by two numbers, $q$ and $n$, where $q$ is the size of the base field and $n$ is the dimension of the vector space. A simple inspection using \cite[Table 1]{lucdam} shows that, if $m(S) = j \le x$, then $q^{n-1} \le x$, and for any $n \le 3$ we have exactly one choice for $q$. This in turn gives a bounded number of choices for $T$ since the table is finite. For $n \ge 4$ we have $q \le x^\frac{1}{3}$ and at most $d \log(x)$ choices for $n$, so, using Lemma \ref{powers} and the Prime Number Theorem, we find at most $d \log(x) x^\frac{1}{3}/\log(x^\frac{1}{3})$ choices for $T$, where $d$ is a constant, giving an upper bound of $3dx^\frac{1}{3}$. This holds for every table entry, and the claim now follows from the fact that the table has a constant number of entries.

By Lemma \ref{simple} (1) there exists a positive constant $d$ such that, when setting $g(x) = d \log(x) x^\frac{1}{3}$, we have $$|\{G \in \mathcal{G} : n_{\sigma}(G) = n\}| \le g(n) \le g(x)$$ for every $n \le x$, observing that there are $o(1) \log(x)$ choices for $m(S)$, and hence $k$ is uniquely determined. We are going to use this function $g(x)$ below when we apply Lemma \ref{g(x)}.

Let $\mathcal{A}$ be the family of the alternating groups $A_n$ with $n \ge 5$, and let $\mathcal{P}$ be the set of simple groups isomorphic to $\PSL(2,q)$ with $q$ a prime power.  As in Lemma \ref{simple} (2), let $\mathcal{S}$ be the family of the nonabelian simple groups not in $\mathcal{P} \cup \mathcal{A}$. Observe that $\mathcal{G}$ is a disjoint union $\bigcup_{i=1}^6 \mathcal{G}_i$ where $$\mathcal{G}_1 := \{G \in \mathcal{G} : k \ge 1,\ T \in \mathcal{S}\},\ \mathcal{G}_2 := \{G \in \mathcal{G} : k = 1,\ T \in \mathcal{A}\},$$
$$\mathcal{G}_3 := \{G \in \mathcal{G} : k = 2,\ T \in \mathcal{A}\},\ \mathcal{G}_4 := \{G \in \mathcal{G} : k \ge 3,\ T \in \mathcal{A}\},$$
$$\mathcal{G}_5 := \{G \in \mathcal{G} : k = 1,\ T \in \mathcal{P}\},\ \mathcal{G}_6 := \{G \in \mathcal{G} : k \ge 2,\ T \in \mathcal{P}\}.$$

By Lemmas \ref{five} and \ref{lem:tom3.2}, $n_{\sigma}(G) \le \ell_G(\soc(G)) \le \sigma(G)$. 

Using Lemma \ref{powers} and Lemma \ref{simple} (2), we see that $$|\{ n_{\sigma}(G)\ :\ G \in \mathcal{G}_1,\ \sigma(G) \le x\}| \le c_1 x^\frac{1}{2}/\log(x).$$Using that $n^3 \le \sigma(A_n), \sigma(S_n)$ for $n$ large (see \cite[Theorem 3.1]{Maroti} and \cite[Theorem 9.2]{lucmar}) and that $\Aut(A_n)= S_n$ for $n$ large, we see that $$|\{ G \in \mathcal{G}_2 : \sigma(G) \le x\}| \le c_2 x^\frac{1}{3}.$$Since $m(A_n)=n$ and $\Aut(A_n \times A_n) = \Aut(A_n) \Wr C_2$, we clearly have $$|\{G \in \mathcal{G}_3 : \sigma(G) \le x\}| \le c_3 x^\frac{1}{2},$$ $$|\{ n_{\sigma}(G) : G \in \mathcal{G}_4,\ \sigma(G) \le x\}| \le c_4 x^\frac{1}{3}.$$If $G \in \mathcal{G}_5$ and $q$ is the size of the base field, then $q \le x$. If $q$ is not a prime, then by the Prime Number Theorem there are at most $o(1) x^\frac{1}{2}$ such $q$, since, if $q=p^f$ with $p$ prime and $f>1$, then $p^2 \le p^f = q$ gives at most $o(1)x^\frac{1}{2}/\log(x)$ choices for $p$ and at most $o(1)\log(x)$ choices for $f$. Using that for $p$ a large prime we have $p^2/2 \le \sigma(\PSL(2,p)) = \sigma(\PGL(2,p))$ (see \cite{BFS}) and that $\Aut(\PSL(2,p)) = \PGL(2,p)$, together with the Prime Number Theorem, we see that $$|\{G \in \mathcal{G}_5 : \sigma(G) \le x\}| \le c_5 x^\frac{1}{2}.$$ Again using the Prime Number Theorem and Lemma \ref{powers}, we have that $$|\{ n_{\sigma}(G) : G \in \mathcal{G}_6,\ \sigma(G) \le x\}| \le c_6 x^\frac{1}{2}/\log(x).$$Combining the above with Lemma \ref{g(x)}, where $g(x)=d \log(x) x^\frac{1}{3}$, we obtain that 
\begin{align*}
\{G \in \mathcal{G} &: \sigma(G) \le x\}|\\ 
                &\le c_1 dx^\frac{1}{2}x^\frac{1}{3}  +c_2 x^\frac{1}{3} +c_3x^\frac{1}{2}+c_4 dx^\frac{1}{3}x^\frac{1}{3} \log(x)+c_5x^\frac{1}{2}+c_6d x^\frac{1}{2}x^\frac{1}{3},
\end{align*}                
which is at most $c x^{5/6}$ with $c$ a constant, completing the proof.
\end{proof}

\begin{proof}[Proof of Corollary \ref{cor:density}]
Since the covering numbers of solvable groups are of the form $q+1$ with $q$ a prime power (by Tomkinson's result \cite{Tomkinson}), we know that there are infinitely many natural numbers that are not the covering number of a solvable group. Let now $G \in \mathscr{F}$ be nonsolvable. Up to replacing $G$ with a suitable $\sigma$-elementary quotient $G_0$ of $G$ such that $\sigma(G)=\sigma(G_0)$, we may assume that $G$ is $\sigma$-elementary. The group $G$ must have a unique minimal normal subgroup $N$, where $N$ is nonabelian; otherwise, if $N$ and $L$ are two minimal normal subgroups of $G$, then $N$ is isomorphic to a subgroup of $G/L$, which is solvable, and $L$ is isomorphic to a subgroup of $G/N$, which is solvable as well, contradicting the fact that $G$ is nonsolvable. Since $G$ is $\sigma$-elementary and $G/N$ is solvable, $G/N$ is cyclic. Moreover, $\Phi(G)=\{1\}$ by Lemma \ref{sigmaprim}. This implies that $G$ is a primitive monolithic group with $G/\soc(G)$ cyclic, and now the result follows by Theorem \ref{densityth}.
\end{proof}

\section{Nonabelian $\sigma$-elementary groups whose covering number is at most 129}
\label{sect:129}

In this section, we prove that any nonabelian $\sigma$-elementary group with covering number at most $129$ is both primitive and monolithic.  When combined with the calculations of Section \ref{sect:tables} and Appendix \ref{sect:calcspecific}, this allows us to determine precisely which integers less than or equal to $129$ are covering numbers of finite groups.  The main theorem in this section is an easy consequence of the following lemmas.

\begin{lem} \label{cur}
Let $G$ be a primitive monolithic group with nonabelian socle $N$. Then there exists a set $\{g_1N,\ldots,g_kN\}$ generating $G/N$ with the property that $$\sigma(\langle g_i, N \rangle) \le \sigma^{\ast}(G) + \omega(|g_iN|_{G/N})$$for every $i \in \{1,\ldots,k\}$, where $|g_iN|_{G/N}$ denotes the order of $g_iN$ in $G/N$ and $\omega(m)$ denotes the number of distinct prime divisors of $m$.
\end{lem}

\begin{proof}
There exists a set $\{g_1N,\ldots,g_kN\}$ generating $G/N$ with the property that $\sigma^{\ast}(G) = \sigma_{\Omega}(G)$, where $\Omega = g_1N \cup \ldots \cup g_kN$. In particular, for $i \in \{1,\ldots,k\}$ we have $\sigma_{g_iN}(G) \le \sigma^{\ast}(G)$. If $H$ is a proper supplement of $N$ in $G$, then $H \cap \langle g_i,N \rangle$ is a proper supplement of $N$ in $\langle g_i,N \rangle$, and therefore $g_iN$ is contained in a union of $\sigma_{g_iN}(G)$ proper subgroups of $\langle g_i,N \rangle$. By Lemma \ref{spancoprime}, in order to cover $\langle g_i,N \rangle$ with proper subgroups it suffices to use a family of proper subgroups covering $g_iN$ and the maximal subgroups of $\langle g_i,N \rangle$ containing $N$. This implies that $$\sigma(\langle g_i,N \rangle) \le \sigma_{g_iN}(G) + \omega(|g_iN|_{G/N}) \le \sigma^{\ast}(G) + \omega(|g_iN|_{G/N}),$$ concluding the proof.
\end{proof}

\begin{lem} \label{2n+1}
Let $n$ be a fixed positive integer. Let $\mathcal{F}$ be the family of monolithic primitive groups $X$ of primitivity degree at most $n$, with nonabelian socle $N$, where $X/N$ is either nonsolvable or cyclic, and where $\sigma^{\ast}(X) \le 2n+1$. If for all $X \in \mathcal{F}$ we have $\sigma(X) < 2 \sigma^{\ast}(X)$, then every nonabelian $\sigma$-elementary group $G$ with $\sigma(G) \le 2n+1$ is primitive and monolithic.
\end{lem}

\begin{proof}
Let $G$ be a nonabelian $\sigma$-elementary group, and let $\soc(G) = N_1 \times \cdots \times N_t$. By Lemma \ref{sigmaprim} we know that at most one of the $N_i$'s is abelian. So we may assume that $N_i$ is nonabelian whenever $i \ge 2$. We need to show that $t=1$, so assume for the purpose of contradiction that $t \ge 2$. Let $X_i$ be the primitive monolithic group associated with $N_i$ for all $i=1,\ldots,t$. By Lemma \ref{sigmastar} we know that $\sum_{i=1}^t \sigma^{\ast}(X_i) \le \sigma(G) \le 2n+1$ for all $i$; in particular, $\sigma^{\ast}(X_i) \le 2n+1$ for all $i$. We consider the two possible cases.

Assume first that $N_1$ is abelian. In this case, since the center of $G$ is trivial by Lemma \ref{lem:trivcenter}, $\sigma(G) < 2|N_1|$ by Lemma \ref{abmns}, and, since $\ell_{X_1}(N_1) = |N_1|$, we have $$\frac{1}{2} \sigma(G)+\ell_{X_2}(N_2) \le \sum_{i=1}^t \sigma^{\ast}(X_i) \le \sigma(G).$$This means $2 \ell_{X_2}(N_2) \le \sigma(G) \le 2n+1$, and therefore $\ell_{X_2}(N_2) \le n$.  Since $X_2$ is a quotient of $G$, we conclude that $X_2/N_2$ is either nonsolvable or cyclic by Lemma \ref{solcyc}, implying $\sigma(X_2) < 2 \sigma^{\ast}(X_2)$ by hypothesis. Since $X_2$ is a quotient of $G$, we have $\sigma(G) \le \sigma(X_2) < 2 \sigma^{\ast}(X_2)$, and, by Lemma \ref{indexbound}, we have $\ell_{X_2}(N_2) \le \sigma^{\ast}(X_2)$.  Combining this with Lemma \ref{abmns} yields 
\[\ell_{X_1}(N_1)+\ell_{X_2}(N_2) \le |N_1|+\sigma^{\ast}(X_2) \le \sigma(G) < \min \{2|N_1|,2 \sigma^{\ast}(X_2)\}.\]
But $|N_1|+\sigma^{\ast}(X_2) < 2|N_1|$ implies $\sigma^{\ast}(X_2) < |N_1|$ and $|N_1|+\sigma^{\ast}(X_2) < 2 \sigma^{\ast}(X_2)$ implies $|N_1| < \sigma^{\ast}(X_2)$, a contradiction.

We may thus assume $N_1$ is nonabelian. In this case, we may assume that \[\min \{\sigma^{\ast}(X_i)\ :\ i=1,\ldots,n\} = \sigma^{\ast}(X_1).\] Therefore, $$2 \ell_{X_1}(N_1) \le 2 \sigma^{\ast}(X_1) \le \sum_{i=1}^t \sigma^{\ast}(X_i) \le \sigma(G) \le 2n+1,$$which implies $\ell_{X_1}(N_1) \le n$, and so, since $X_1$ is a quotient of $G$, we have that $X_1/N_1$ is either nonsolvable or cyclic by Lemma \ref{solcyc}, we have $\sigma(X_1) < 2 \sigma^{\ast}(X_1)$ by hypothesis. Hence $$t \cdot \sigma^{\ast}(X_1) \le \sum_{i=1}^t \sigma^{\ast}(X_i) \le \sigma(G) \le \sigma(X_1) < 2 \sigma^{\ast}(X_1),$$which contradicts the fact that $t \ge 2$, completing the proof.
\end{proof}

\begin{rem}
\label{rem:DL}
 If $\sigma(X) < 2\sigma^{\ast}(X)$ for all primitive monolithic groups $X$ with a nonabelian socle, then Lemma \ref{2n+1} implies that Conjecture \ref{conj:DL} is true.
\end{rem}

\begin{lem} \label{2star}
Let $X$ be a primitive monolithic group with nonabelian socle $N$. If $X/N$ is a cyclic $p$-group for some prime $p$ then $\sigma(X) \le \sigma^{\ast}(X)+1 < 2 \sigma^{\ast}(X)$.
\end{lem}

\begin{proof}
Since $X/N$ is a cyclic $p$-group, it admits exactly one maximal subgroup. Therefore a union $\Omega$ of cosets of $N$ in $X$ generates $X$ if and only if it contains a coset $xN$, where $x$ does not belong to the unique maximal subgroup of $X$ containing $N$. It follows that there exists such $x$ with $\sigma^{\ast}(X) = \sigma_{xN}(X)$. Observe that since $X/N$ is a $p$-group, we may choose such an $x$ of $p$-power order. Now we can cover $xN$ with a family $\mathcal{K}$ consisting of $\sigma^{\ast}(X)$ supplements of $N$, which therefore cover all the cosets $x^kN$ with $k$ coprime to $p$ by Lemma \ref{spancoprime}. What is left to cover is every coset $x^{pk}N$ for $k \ge 1$. Thus adding $\langle N,x^p \rangle \neq X$, we conclude that $\sigma(X) \le \sigma^{\ast}(X)+1$.
\end{proof}

Our main result in this section is now an easy consequence of the above lemmas.

\begin{thm}
\label{thm:129mono}
Let $G$ be a nonabelian $\sigma$-elementary group with $\sigma(G) \le 129$. Then $G$ is primitive and monolithic with primitivity degree at most $129$.
\end{thm}

\begin{proof}
We show that $G$ is primitive and monolithic. By Lemma \ref{2n+1}, to do so it is enough to show that $\sigma(X) < 2 \sigma^{\ast}(X)$ whenever $X$ is a primitive monolithic group of degree at most $64$ satisfying each of the following three conditions: (1) $X$ has nonabelian socle $N$, (2) $X/N$ is either nonsolvable or cyclic, and (3) $\sigma^{\ast}(X) \le 129$. Let $X$ be such a group.  If $X/N$ is a cyclic $p$-group for some prime $p$, then Lemma \ref{2star} implies $\sigma(X) < 2 \sigma^{\ast}(X)$.  Now assume $X/N$ is not a cyclic $p$-group. A GAP check shows that the only possibility is $X \cong \Aut(\PSL(2,27))$, in which case $X/N \cong C_6$ and $\ell_X(N)=28$. In this case, Lemma \ref{cur} implies that either $\sigma(X) \le \sigma^{\ast}(X) + 2 < 2 \sigma^{\ast}(X)$, or, for one of the $g_i$'s in this lemma, $\langle N,g_i \rangle \cong \PGL(2,27)$, and so $\sigma^{\ast}(X) \ge \sigma(\PGL(2,27)) - 1 = 378$ holds, a contradiction to $\sigma^{\ast}(X) \le 129$. Thus $\sigma(X) < 2 \sigma^{\ast}(X)$. Lemma \ref{indexbound} implies that the smallest primitivity degree of $G$ is at most $\sigma(G)$.
\end{proof}

\section{Computational methods}
\label{sect:comp}

In this section, we outline the computational methods used to prove Theorem \ref{thm:129}.  By Tomkinson's result (see Proposition \ref{prop:solvable} below), it suffices to consider nonsolvable $\sigma$-elementary groups, and by Theorem \ref{thm:129mono}, any nonabelian $\sigma$-elementary group $G$ with $\sigma(G) \le 129$ is primitive and monolithic with a primitivity degree of at most $129$.   Using $\GAP$, we are able to list every nonsolvable primitive group with degree of primitivity at most $129$.  The covering number of many of these groups is known; see Section \ref{sect:known} below.  Moreover, the covering number of affine general linear groups and affine special linear groups when $n \ge 3$ are determined in Section \ref{sect:affine}.  

All remaining groups, i.e., those groups not explicitly discussed in Sections \ref{sect:known} and \ref{sect:affine}, are listed in Tables \ref{tbl:main1}, \ref{tbl:main2}, \ref{tbl:main3}, \ref{tbl:main4}, and \ref{tbl:main5}, along with a reference as to how the computation was completed for each group.  In many cases, the group has a noncyclic solvable homomorphic image whose covering number is the same as the original group.  In these cases, the homomorphic image is listed in the reference column.  For many primitive groups of affine type -- that is, those that have a unique elementary abelian minimal normal subgroup -- a result due to Detomi and Lucchini (also proved independently by S. M. Jafarian Amiri \cite{Amiri}) can be used:  if $G$ is such a primitive group with elementary abelian minimal normal subgroup $N$ and point stabilizer $H$ in the primitive action and $\sigma(H) < |N|$, then $\sigma(G) = \sigma(H)$; see Lemma \ref{lem:Amiri}.

The covering number of many other groups can be computed exactly using either linear programming methods or other computational techniques.  The details are discussed in Subsections \ref{subsect:LP} and \ref{subsect:verify}, respectively.  There are only a few groups whose covering number cannot be determined using these methods, and they are considered on an ad hoc basis in Appendix \ref{sect:calcspecific}.  

\subsection{Linear programming methods}
\label{subsect:LP}

In \cite{KNS}, the authors created a program in GAP \cite{GAP} that takes as input a group $G$, a list $\mathcal{E}$ of 
elements of $G$, a list $\mathcal{M}$ of maximal subgroups of $G$, and the name of a file of type .lp to which output is written.  This output file is read by the linear optimization software GUROBI \cite{Gu}, which then determines the least number of subgroups conjugate to one of the subgroups in $\mathcal{M}$ needed to cover the elements conjugate to the elements of $\mathcal{E}$.  This function is referred to in the remainder of the paper as ``Algorithm \hyperref[subsect:LP]{KNS},'' and the GAP code for this program can be found in \cite{KNS}. 

For a group of order approximately 500000 or less, Algorithm \hyperref[subsect:LP]{KNS} generally will return a .lp file within 24 hours.  The calculations done here were completed with a laptop that has a Core i7 processor and 16 GB of RAM. The optimization software GUROBI sometimes is able to determine the exact covering number within seconds; other times, the program runs out of memory, but is still able to provide good bounds.  For instance, the previous bounds on the covering number of $J_2$ were $380 \le \sigma(J_2) \le 1220$, given in \cite{HolmesMaroti}.  With the aid of Algorithm \hyperref[subsect:LP]{KNS} and GUROBI, we have improved these bounds to $1063 \le \sigma(J_2) \le 1121$.

\subsection{A verification method for minimal covers and a greedy algorithm}
\label{subsect:verify}

Mar\'{o}ti introduced the following technique for showing that a cover is minimal.  Following \cite{Maroti}, if $\Pi \subseteq G$, we define $\sigma(\Pi)$ to be the least integer $m$ such that $\Pi$ is a subset of the set-theoretic union of $m$ subgroups of $G$; clearly, $\sigma(\Pi) \le \sigma(G)$.  A set $\mathcal{H} = \{H_1, \dots, H_m\}$ of $m$ proper subgroups of $G$ is \textit{definitely unbeatable} on $\Pi$ if both of the following conditions hold:
\begin{itemize}
 \item[(i)] the elements of $\Pi$ are partitioned among the subgroups in $\mathcal{H}$, and
 \item[(ii)] for all subgroups $K \le G$ that are not contained in $\mathcal{H}$, we have $|K \cap \Pi| \le |H_i \cap \Pi|$ for each $i$, $1 \le i \le m$.
\end{itemize}
If $\mathcal{H}$ is definitely unbeatable on $\Pi$, then $|\mathcal{H}| = \sigma(\Pi) \le \sigma(G)$.

However, definite unbeatability is often too stringent a condition.  With this in mind, a more complicated but more generally applicable condition was introduced in \cite{Swartz}.  The following lemma is a slight modification of that condition (in that the parameter $c(M)$ may equal $1$ here) and is useful in cases when the minimal cover is not unique.

\begin{lem}
\label{lem:keylemma}
Let $\Pi$ be a union of conjugacy classes of elements of $G$; let $I \subseteq I_G$, where $I_G$ is an index set for the conjugacy classes of maximal subgroups of $G$; and let $\mathcal{C} = \bigcup_{i \in I} \mathcal{M}_i$ be a cover of $\Pi$ such that each $\mathcal{M}_i$ denotes a conjugacy class of maximal subgroups, the elements of $\Pi$ are partitioned among the subgroups in $\mathcal{C}$, and each subgroup in $\mathcal{C}$ contains elements of $\Pi$.  For a maximal subgroup $M \not\in \mathcal{C}$, define \[c(M) := \sum\limits_{i \in I}\frac{|M \cap \Pi_i|}{|M_i \cap \Pi_i|},\] where $M_i$ is a maximal subgroup in $\mathcal{M}_i$.  If $c(M) \le 1$ for all maximal subgroups $M \not\in \mathcal{C}$, then $\mathcal{C}$ is a minimal cover of the elements of $\Pi$.  Moreover, if $c(M) < 1$ for all maximal subgroups $M \not\in \mathcal{C}$, then $\mathcal{C}$ is the unique minimal cover of the elements of $\Pi$ that uses only maximal subgroups.
\end{lem}

\begin{proof}
The statement with $c(M) < 1$ was proved in \cite{Swartz}; the proof here is nearly identical, save for some strict inequalities being changed to allow for equalities, but it is included for the sake of completeness.  Let $\mathcal{C}$ and $\Pi$ be as in the statement of the lemma, and assume that $c(M) \le 1$ for all maximal subgroups not in $\mathcal{C}$.  Suppose that $\mathcal{B}$ is another cover of the elements of $\Pi$, and let $\mathcal{C}' = \mathcal{C} \backslash (\mathcal{C} \cap \mathcal{B})$ and $\mathcal{B}' = \mathcal{B} \backslash (\mathcal{C} \cap \mathcal{B})$.  The collection $\mathcal{C}'$ consists only of subgroups from classes $\mathcal{M}_i$, where $i \in I$, and we let $a_i$ be the number of subgroups from $\mathcal{M}_i$ in $\mathcal{C}'$.  Similarly, the collection $\mathcal{B}'$ consists only of subgroups from classes $\mathcal{M}_j$, where $j \not\in I$, and we let $b_j$ be the number of subgroups from $\mathcal{M}_j$ in $\mathcal{B}'$.  Note that, since $\mathcal{B}$ is a different cover, for some $j \not\in I$, we have $b_j > 0$.  

By removing $a_i$ subgroups in class $\mathcal{M}_i$ from $\mathcal{C}$, the new subgroups in $\mathcal{B}'$ must cover the elements of $\Pi$ that were in these subgroups.  Hence, for all $i \in I$,  if $M_k$ denotes a subgroup in class $\mathcal{M}_k$ for each $k$, $$a_i |M_i \cap \Pi_i| \le \sum\limits_{j \not\in I} b_j |M_j \cap \Pi_i|,$$ which in turn implies that, for all $i \in I$, we have $$a_i \le \sum\limits_{j \not\in I} b_j \frac{|M_j \cap \Pi_i|}{|M_i \cap \Pi_i|}.$$  This means that
\begin{align*}
|\mathcal{C}'| &= \sum\limits_{i \in I} a_i \le \sum\limits_{i\in I}\sum\limits_{j \not\in I} b_j \frac{|M_j \cap \Pi_i|}{|M_i \cap \Pi_i|} = \sum\limits_{j \not\in I}\sum\limits_{i \in I} b_j \frac{|M_j \cap \Pi_i|}{|M_i \cap \Pi_i|}\\
&= \sum\limits_{j \not\in I}\left(\sum\limits_{i \in I} \frac{|M_j \cap \Pi_i|}{|M_i \cap \Pi_i|}  \right)b_j = \sum\limits_{j \not\in I} c(M_j)b_j \le \sum\limits_{j \not\in I} b_j = |\mathcal{B}'|,
\end{align*}
which shows that $$|\mathcal{C}| = |\mathcal{C}'| + |\mathcal{C} \cap \mathcal{B}| \le |\mathcal{B}'| + |\mathcal{C} \cap \mathcal{B}| = |\mathcal{B}|.$$ Hence, any other cover of the elements of $\Pi$ using only maximal subgroups contains at least as many subgroups as $\mathcal{C}$.  Therefore, $\mathcal{C}$ is a minimal cover of the elements of $\Pi$. 
\end{proof}

\begin{algorithm}
  \caption{\textsc{CoveringNumberBounds}}\label{alg}
  \begin{algorithmic}[1]
   \Require A finite group $G$.
   \Ensure A triple $(\ell, u, c)$, where $\ell \le \sigma(G) \le u$ and $c$ is \textsc{True} if it is verified that $\sigma(G) = u$ and \textsc{False} otherwise.
  \State $max$ := list of representatives of each class of maximal subgroups of $G$
  \State $eltM$ := for each subgroup $M$ in $max$, a list of representatives of each conjugacy class 
  \indent of elements of $M$
  \State $conj$ := list of nonidentity conjugacy classes of elements of $G$
  
  \State $u$ := 0
  \State $minlist$ := an empty list
  \State $cvalues$ := list with every entry $0$ of length the size of $max$ 
  
  \While{$conj$ is nonempty}
    \State $elts$ := for each class $x^G$ left in $conj$, the elements of $eltM$ that are in $x^G$
    \State $ints$ := for each class $x^G$ left in $conj$, a list of the sizes of the intersection of $x^G$ with 
    \indent each subgroup $M$ in $max$, created using the list $elts$ by summing the sizes of the 
    \indent conjugacy classes in $M$ over the set of elements in $eltM$ that are in $x^G$
    \State $mins$ := for each class $x^G$ left in $conj$, the minimum number of subgroups needed 
    \indent to cover $x^G$, calculated by dividing the size of $x^G$ by the maximum intersection size 
    \indent from $ints$ corresponding to $x^G$
    \State $best$ := maximum of $mins$, which can be thought of as the minimum number of 
    \indent subgroups needed at this stage to get a cover
    \State \textbf{add} $best$ to $minlist$
    \State $x_0^G$:= the conjugacy class in $conj$ that needed $best$ subgroups to be covered
    \State $M_0$ := a maximal subgroup from $max$ from the class used to cover $x_0^G$ with $best$    \indent subgroups 
    \State $cvalueupdate$ := list with entry $|M \cap x_0^G|/|M_0 \cap x_0^G|$ for each $M \in max$
    \State $cvalue$:= $cvalue$ + $cvalueupdate$ (addition is entrywise)
    \If{ $best \neq |G:M_0|$ }
    \State $c$ := \textsc{False}
    \EndIf
    \State $u$ := $u + |G:M_0|$
    \State $conj$ := any remaining conjugacy classes that do not intersect $M_0$
  \EndWhile
  \State $\ell$ := the first entry in $minlist$
  \If{ $c = \textsc{True}$ }
  \For{ $i$ in $cvalue$ }
  \If { $i > 1$ }
  \State $c := \textsc{False}$
  \EndIf
  \EndFor
  \EndIf
  \Return $(\ell, u, c)$
  \end{algorithmic}
\end{algorithm}

In practice, there is often a union of conjugacy classes $\Pi$ of elements of $G$ and a minimal cover $\mathcal{C}$ of the elements of $\Pi$ that satisfies the hypotheses of Lemma \ref{lem:keylemma}.  We can design an algorithm exploiting this idea that works roughly as follows:  each conjugacy class of elements and representatives for each conjugacy class of maximal subgroups are computed in $\GAP$.  Next, the conjugacy class $x^G$ of elements that requires the most maximal subgroups to cover is determined.  Greedily, we take as part of a cover all subgroups from a conjugacy class $\calM$ of maximal subgroups that most efficiently covers $x^G$.  All elements that are covered by the subgroups of $\calM$ are removed, and this process is repeated again and again until all elements are covered.  Often, the cover produced this way is a minimal cover, and this can typically be verified by using Lemma \ref{lem:keylemma}.  Even if the cover is not verifiably minimal, the function returns upper and lower bounds for $\sigma(G)$.  The steps of this procedure are listed in Algorithm \ref{alg}, which is written in pseudocode. 

We remark that, while it would be ``simpler'' to calculate $ints$ in Step 10 of Algorithm \ref{alg} by taking the intersection of class $x^G$ with each subgroup in $max$, for many groups the sizes of the conjugacy classes are quite large, and it is much faster for such groups to calculate the intersection sizes as described in the pseudocode.  Algorithm \ref{alg} can also be altered to return additional information, such as the classes $x_0^G$ and subgroups $M_0$ chosen in various iterations of the \textbf{while} loop, which is useful for ad hoc calculations like those in Appendix \ref{sect:calcspecific}.

\section{Known bounds on and values of covering numbers}
\label{sect:known}

We collect in this section a list of known results regarding the covering number of specific families of groups.  We use the notation $S_n$ to refer to the symmetric group of degree $n$ and $A_n$ to refer to the alternating group of degree $n$.  The first proposition combines the results of Cohn and Tomkinson and completely solves the problem of which integers are covering numbers of solvable groups. In the tables we also indicate the smallest primitivity degree $m(G)$ of any given primitive group $G$.

\begin{prop}
 \label{prop:solvable}
\begin{enumerate}[(i)]	
 \item \cite[Corollary to Lemma 17]{Cohn} For every prime $p$ and positive integer $d$, there exists a group $G$ with covering number $p^d + 1$.
 \item \cite[Theorem 2.2]{Tomkinson} Let $G$ be a finite solvable group and let $H/K$ be the smallest chief factor of $G$ having more than one complement in $G$.  Then $\sigma(G) = |H/K| + 1$.  In particular, the covering number of any (noncyclic) solvable group has the form $p^{d} + 1$, where $p$ is a prime and $d$ is a positive integer. 
\end{enumerate}
\end{prop}

The following table summarizes what is currently known about covering numbers of symmetric groups.

\begin{center}
 \begin{table}[H]
	
	\centering
  \begin{tabular}{c|c|c|c}
    Group & $m(G)$ & Covering Number & Citation \\
    \hline
    $S_5$ & $5$ & $16$ & \cite{Cohn}\\
    $S_6$ & $6$ & $13$ & \cite{AAS}\\
    $S_8$ & $8$ & $64$ & \cite{KNS}\\
    $S_9$ & $9$ & $256$ & \cite{KNS}\\
    $S_{10}$ & $10$ & $221$ & \cite{KNS}\\
    $S_{12}$ & $12$ & $761$ &\cite{KNS}\\
    $S_{14}$ & $14$ & $3096$ & \cite{OppenheimSwartz}\\
    $S_{18}$ & $18$ & $36773$ & \cite{Swartz}\\
    $S_{6k}, k \ge 4$ & $6k$ & $\frac{1}{2} {{6k} \choose {3k}} + \sum_{i=0}^{2k - 1}\limits {{6k} \choose {i}}$ & \cite{Swartz}\\
    $S_{2k+1}, k \neq 4$ & $2k+1$ & $2^{2k}$ & \cite{Maroti}\\
    $S_{2k}$, $k \ge 16$ & $2k$ & $> \frac{1}{2}{{2k} \choose k}$ & \cite{Maroti}\\
  \end{tabular}
	\caption{Covering numbers of symmetric groups}
	\label{tbl:Sn}
\end{table}
\end{center}



The following table summarizes what is currently known about covering numbers of alternating groups.

 \begin{table}[H]
	\centering
  \begin{tabular}{c|c|c|c}
   Group & $m(G)$ & Covering Number & Citation\\
   \hline
   $A_5$ & $5$ & $10$ & \cite{Cohn}\\
   $A_6$ & $6$ & $16$ & \cite{BFS}\\
   $A_7$ & $7$ & $31$ & \cite{KappeRedden}\\
   $A_8$ & $8$ & $71$ & \cite{KappeRedden}\\
   $A_9$ & $9$ & $157$ & \cite{EMN}\\
   $A_{10}$ & $10$ & $256$ & \cite{Maroti}\\
   $A_{11}$ & $11$ & $2751$ & \cite{EMN}\\
   $A_{4k+2}$ & $4k+2$ & $2^{4k}$ & \cite{Maroti}\\
   $A_n$, $n \ge 12$ & $n$ & $\ge 2^{n-2}$ & \cite{Maroti}
  \end{tabular}
	\caption{Covering numbers of alternating groups}
	\label{tbl:An}
 \end{table}

%

The following table summarizes what is currently in the literature regarding covering numbers of projective linear groups of dimension $2$.

\begin{table}[H]
	\centering
  \begin{tabular}{c|c|c|c}
  Group & $m(G)$ & Covering Number & Citation\\
  \hline
  $\PSL(2,5)$ & $6$ & $10$ & \cite{Cohn}\\
  $\PGL(2,5)$ & $6$  & $16$ & \cite{Cohn}\\
  $\PSL(2,7)$ & $7$ & $15$ & \cite{BFS}\\
  $\PGL(2,7)$ & $8$ & $29$ & \cite{BFS}\\
  $\PSL(2,9)$ & $10$ & $16$ & \cite{BFS}\\
  $\PGL(2,9)$ & $10$ & $46$ & \cite{BFS}\\
  $\PGammaL(2,8)$ & $9$ & $29$ & \cite{Garonzi3}\\
  $\PSL(2,q)$, $\PGL(2,q)$, $q \ge 8$ even & $q+1$ & $\frac{1}{2}q(q+1)$ & \cite{BFS}\\
  $\PSL(2,q)$, $\PGL(2,q)$, $q > 9$ odd & $q+1$ & $\frac{1}{2}q(q+1) + 1$ & \cite{BFS}\\
  \end{tabular}
     \caption{Covering numbers of $2$-dimensional linear groups}
     \label{tbl:linear2}
\end{table}



By \cite{Lucido}, if $q = 2^{2m+1}$ for some $m \in \N$, then $\sigma(\Sz(q)) = \frac{1}{2}q^2(q^2 + 1)$.  


The following table summarizes what is currently in the literature regarding covering numbers of sporadic simple groups with a degree of primitivity less than or equal to $129$.  Using Algorithm \hyperref[subsect:LP]{KNS} and GUROBI \cite{Gu}, we have improved the bounds for $J_2$ to $1063 \le \sigma(J_2) \le 1121$.

\begin{table}[H]
\centering
\begin{tabular}{c|c|c|c}
Group & $m(G)$ & Covering Number & Citation\\
\hline
$M_{11}$ & $11$ & $23$ & \cite{Holmes}\\
$M_{12}$ & $12$ & $208$ & \cite{KNS}\\
$M_{22}$ & $22$ & $771$ & \cite{Holmes}\\

$M_{23}$ & $23$ & $41079$ & \cite{Holmes}\\
$M_{24}$ & $24$ & $3336$ & \cite{EpsteinMagliveras}\\

$HS$ & $100$ & $1376$ & \cite{HolmesMaroti}\\

$J_2$ & $100$ & $\ge 380$ & \cite{HolmesMaroti}\\

\end{tabular}
\caption{Covering numbers of various sporadic simple groups}
\label{tbl:sporadic}
\end{table}

The following result was the main application of Lemma \ref{lem:Amiri} in \cite{Amiri} and proves results about $2$-dimensional affine general linear groups.

\begin{lem}\cite{Amiri}
Let $p>3$ be a prime.  Then $\sigma(\AGL(2,p)) = p(p+1)/2 + 1.$ 
\end{lem}

However, when combined with the results about $\PSL(2,q)$ and $\PGL(2,q)$ when $q$ is not a prime (see Table \ref{tbl:linear2}), Lemma \ref{lem:Amiri} can be used to prove the following stronger result.

\begin{lem}
\label{lem:agl2}
Let $q$ be a prime power, $q \ge 4$.  Then $\sigma(\AGL(2,q)) = \sigma(\ASL(2,q)) = \sigma(\PSL(2,q))$, and consequently $\AGL(2,q)$ and $\ASL(2,q)$ are never $\sigma$-elementary when $q \ge 4$.
\end{lem}

Finally, we have the following result about the covering numbers of two other primitive groups.

\begin{prop}\cite{Garonzi1}
\label{prop:A5prod}
For $A_5 \Wr 2$ and $(A_5 \times A_5):4$ we have:
\begin{enumerate}[(i)]
 \item $\sigma(A_5 \Wr 2) = 57$;
 \item $\sigma((A_5 \times A_5):4) = 126$, where $(A_5 \times A_5):4$ is the preimage of the normal cyclic subgroup of order $4$ in $D_8$ via $\Aut(A_5 \times A_5) \to \Out(A_5 \times A_5) \cong D_8$. 
\end{enumerate}
\end{prop}

\section{Covering affine general linear groups and affine special linear groups with proper subgroups}
\label{sect:affine}

This section is dedicated to proving Theorem \ref{thm:AGL}, which can be interpreted as a generalization of Cohn's result \cite{Cohn} listed in Proposition \ref{prop:solvable}.

\begin{proof}[Proof of Theorem \ref{thm:AGL}]
Since the proof is analogous for $\ASL(n,q)$, we will only show the result in each case for $\AGL(n,q)$.

First, if $n = 1$, then we have $\sigma(\AGL(1,q)) = q+1 = (q^2 - 1)/(q-1)$.  This follows from \cite[Lemma 2.1]{Tomkinson}. 

When $n \ge 3$, we consider the group $\AGL(n,q)$. By Lemma \ref{abmns}, \[\sigma(\AGL(n,q)) \le (q^{n+1} - 1)/(q-1).\]  It remains to show that there is no smaller cover.

Let $\AGL(n,q) = V \rtimes H$, where $V$ is the underlying vector space over $\GF(q)$ (defined additively) and $H$, which is isomorphic to $\GL(n,q)$, is the stabilizer of $0 \in V$.  Define the cover $\C = \C_1 \cup \C_2$ to be the union of two classes $\C_1$ and $\C_2$ of maximal subgroups.  Let $\C_1$ consist of all point stabilizers of $\AGL(n,q)$ in its natural primitive action on $q^{n}$ points; that is, $\C_1$ consists of all conjugates of $H$ in $\AGL(n,q)$.  In this case, $|\C_1| = q^{n}$.  Let $\C_2$ consist of all subgroups isomorphic to $V \rtimes K$, where $V$ is the unique minimal normal subgroup of $\AGL(n,q)$ of order $q^{n}$ and $K$ is the stabilizer of a 1-dimensional subspace.  Since there are $(q^{n} - 1)/(q-1)$ lines through the origin, there are $(q^{n} - 1)/(q-1)$ such groups in $\C_2$.  This implies that $|\C| = (q^{n+1} - 1)/(q-1)$.  By Lemma \ref{abmns}, the collection $\C$ is a cover of the elements of $\AGL(n, q)$.  

We will show that no smaller collection of subgroups can cover the elements of $\AGL(n,q)$.  First, by a result of Kantor \cite{Kantor}, the only maximal subgroups of $\GL(n,q)$ containing Singer cycles are field extension groups isomorphic to $\GL(n/b,q^b)$, where $b > 1$ is a divisor of $n$.   Moreover, the Singer cycles are partitioned among these subgroups.  Since \[q^{n(n-1)}(q-1)^n = \prod_{i=0}^{n-1} (q^n - q^{n-1})\le |\GL(n,q)| = \prod_{i=0}^{n-1}(q^n - q^i) \le \prod_{i=0}^{n-1} q^n = q^{n^2},\] we have
\begin{align*}
\sigma(\GL(n,q)) &\ge \frac{|\GL(n,q)|}{|\GL(n/b,q^b)|} \ge \frac{q^{n(n-1)}(q-1)^n}{\left(q^{b}\right)^{(n/b)^2}}\\ 
                 &= q^{n^2\left(1 - \frac{1}{b}\right) - n}(q-1)^n \ge q^{\frac{n^2}{2} - n}(q-1)^n \ge \frac{q^{n+1} - 1}{q-1},
\end{align*}
when $n \ge 4$ and when $n=3$ and $q \ge 3$.  When $n = 3$ and $q=2$, we have $\GL(3,2) \cong \PSL(2,7)$, and hence $\sigma(\GL(3,2)) = \sigma(\PSL(2,7)) = 15$ by Table \ref{tbl:linear2}.  By Lemma \ref{lem:Garonzi}, the groups in $\C_1$ must appear in any minimal cover of the elements of $\AGL(n,q)$.

In $V$, there is a natural bijection between $1$-dimensional subspaces and hyperplanes, and so for any nonzero $v \in V$ we define $\phi(v)$ to be the hyperplane of $V$ such that $V = \langle v \rangle \oplus \phi(v)$.  Let $v \in V$, $v \neq 0$, and, if $s$ is a Singer cycle on $\phi(v)$ that fixes both $0 \in \phi(v)$ and $v$, then define $g$ to be element of $\AGL(n,q)$ that corresponds to $s$ followed by a translation by $v$.  The element $g$ fixes no points of $V$.  To see this, consider $w \in V$.  Since $\phi(v)$ complements $\langle v \rangle$, there exist a unique $a \in \GF(q)$ and $u \in \phi(v)$ such that $w = av + u$.  If $w^g = w$, this implies that $(a+1)v + u^s = av + u,$ that is, we have $v = u - u^s \in \phi(v)$, a contradiction.  Hence $g$ is not contained in any group in $\C_1$.

Moreover, $|g| = p \cdot (q^{n-1} - 1)/(q-1)$, since translation by $v$ has order $p$ and the Singer cycle on $\phi(v)$ fixes $v$ and hence commutes with translation by $v$.  For nonzero vectors $v_1, v_2 \in V$, let $g_1$ be an element that corresponds to a Singer cycle on the hyperplane $\phi(v_1)$ followed by translation by $v_1$, and let $g_2$ be an element that corresponds to a Singer cycle on the hyperplane $\phi(v_2)$ followed by translation by $v_2$. Then $g_1^p$ and $g_2^p$ are Singer cycles of hyperplanes, and both $g_1^p$ and $g_2^p$ are elements of $H$.  By \cite[Theorem 4.1 (2)]{BEGHM1} (see also \cite{BEGHM2}), if $n \ge 4$, $(n,q) \neq (4,2), (11,2)$, and a maximal subgroup $M$ of $H$ contains $g_1^p$ and $g_2^p$, then both $g_1^p$ and $g_2^p$ leave the same $1$-dimensional subspace (and hyperplane) fixed; that is, $\langle v_1 \rangle = \langle v_2 \rangle$.  When $n = 3$, we can derive the same result using \cite[Tables 8.3-8.4]{BrayHoltRoneyDougal} by considering a primitive prime divisor of $q^3 - 1$ when $q \neq 4$.  Furthermore, when $(n,q) = (3,4), (4,2)$, the result follows by computation in $\GAP$, and, when $(n,q) = (11,2)$, the result follows by considering \cite[Tables 8.70-8.71]{BrayHoltRoneyDougal}.  We deduce from the above that, if $\langle v_1 \rangle \neq \langle v_2 \rangle$, then $\langle g_1^p, g_2^p \rangle = H$. Since $g_1$ fixes no points of $V$ and $H$ is the stabilizer of $0 \in V$, it follows that $g_1 \not\in H$.  Because $H$ is a maximal subgroup of $\AGL(n,q)$ and $g_1 \not\in H$, this implies that
\[ \AGL(n,q) = \langle g_1, H \rangle = \langle g_1, g_2^p \rangle \le \langle g_1, g_2 \rangle,\]
and in this case $g_1$ and $g_2$ generate all of $\AGL(n,q)$.  Hence, if $g_1$ and $g_2$ do not stabilize the same $1$-dimensional subspace, they pairwise generate all of $\AGL(n,q)$.  Therefore, we need at least as many subgroups as there are $1$-dimensional subspaces to cover the elements of this type, i.e., we need at least $(q^n - 1)/(q-1)$ subgroups in addition to those from $\C_1$, and the result follows.
\end{proof}

It is unclear which integers of the form $(q^3 - 1)/(q-1) = q^2 + q + 1$, where $q$ is a prime power, are covering numbers.  What is known so far is summarized in Table \ref{tbl:q^2+q+1}, and there does not appear to be any clear pattern.  The smallest open case is when $q=11$, and no group has been found yet with covering number $133 = 11^2 + 11 + 1$. 

\begin{center}
\begin{table}[H]\centering
 \begin{tabular}{c|c|c|c}
  $q$ & $q^2 + q + 1$ & Covering number?& $\sigma$-elementary groups\\
  \hline
  $2$ & $7$  & No & $\varnothing$\\
  $3$ & $13$ & Yes & $S_6$ \\
  $4$ & $21$ & No & $\varnothing$ \\
  $5$ & $31$ & Yes & $A_7$, $\AGL(4,2)$\\
  $7$ & $57$ & Yes & $A_5 \Wr 2$ \\
  $8$ & $73$ & Yes &$(A_6 \times A_6):4$ \\
  $9$ & $91$ & No & $\varnothing$\\
  $11$ & $133$ & ?& ? \\
 \end{tabular}
\caption{Integers of the form $q^2 + q + 1$ and whether or not they are covering numbers}
\label{tbl:q^2+q+1}
\end{table}
\end{center}

\section{Tables and computational results}
\label{sect:tables}

The purpose of this section is to provide a summary of the calculations of the covering numbers of the primitive monolithic groups with a degree of primitivity of at most $129$.  We remark that Theorem \ref{thm:129} follows from Theorem \ref{thm:129mono}, Proposition \ref{prop:solvable}, and Table \ref{tbl:sigmaelementary}.  Table \ref{tbl:sigmaelementary} contains the complete list of nonsolvable $\sigma$-elementary groups $G$ where $\sigma(G) \le 129$, which summarizes the information about nonsolvable $\sigma$-elementary groups.  Table \ref{tbl:sigmaelementary} follows from the known results in Section \ref{sect:known} and the results of calculations which are listed in Tables \ref{tbl:main1}, \ref{tbl:main2}, \ref{tbl:main3}, and \ref{tbl:main4}.

\begin{center}
\begin{table}[H]\centering
\begin{tabular}{c|c}
Covering Number & Nonsolvable $\sigma$-elementary groups\\
\hline
10 & $A_5$\\
13 & $S_6$\\
15 & $\PSL(2,7)$\\
16 & $S_5$, $A_6$\\
23 & $M_{11}$\\
29 & $\PGL(2,7)$, $\PGammaL(2,8)$\\
31 & $A_7$, $\AGL(4,2)$\\
36 & $\PSL(2,8)$\\
40 & $\ASL(3,3)$, $\AGL(3,3)$\\ 
46 & $M_{10}$, $\PGL(2,9)$\\
57 & $A_5 \Wr 2$\\
60 & $\PGammaU(3,3)$\\
63 & $\AGL(5,2)$\\
64 & $S_7$, $S_8$, $\PSU(3,3)$, $\PSp(4,3):2$, $\Sp(6,2)$\\
67 & $\PSL(2,11)$, $\PGL(2,11)$, $\PSp(4,3)$\\
71 & $A_8$\\
73 & $(A_6 \times A_6):4$\\
85 & $\ASL(3,4)$, $\AGL(3,4)$, $\ASigmaL(3,4)$\\
86 & $\PGammaL(2,16)$, $\PSL(2,16).2$\\
92 & $\PSL(2,13)$, $\PGL(2,13)$\\
114 & $\PSL(2,7) \Wr 2$\\
121 & $\ASL(4,3)$, $\AGL(4,3)$\\
126 & $(A_5 \times A_5):4$\\
127 & $\AGL(6,2)$, $\PSigmaL(2,25)$\\

\end{tabular}
\caption{The nonsolvable $\sigma$-elementary groups $G$ with $\sigma(G) \le 129$}
\label{tbl:sigmaelementary}
\end{table}
\end{center}

The following tables list groups and their covering numbers, along with references.  Excluded are groups whose covering number was determined previously.  Specifically, we have excluded the groups $S_n$, $A_n$, $\PSL(2,q)$, $\PGL(2,q)$, $\PGammaL(2,8)$, $\Sz(q)$, $M_{11}$, $M_{12}$, $M_{22}$, $M_{23}$, $M_{24}$, $HS$, $\AGL(2,q)$, $A_5 \Wr 2$, and $(A_5 \times A_5):4$, since these are dealt with in Section \ref{sect:known}.  In the reference column, when a specific group $H$ is listed, it means that the group has $H$ as a homomorphic image and the same covering number as $H$.  For instance, ``$S_3$'' means the group projects onto the symmetric group $S_3$ and has covering number $4$, since $\sigma(S_3) = 4$.  If it says ``Algorithm \hyperref[subsect:LP]{KNS}'' in the reference column, this means that Algorithm \hyperref[subsect:LP]{KNS} was used with representatives of all conjugacy classes of elements and representatives of all conjugacy classes of maximal subgroups as inputs to generate a .lp file that was then optimized using GUROBI \cite{Gu}.  Different groups in $\GAP$ can be given the same name; for instance, $(A_6 \times A_6).2^2$ means some extension of $A_6 \times A_6$ by a Klein $4$-group.  When it says ``(all such groups)'' in a table, we mean that all such groups listed by $\GAP$ with that name and degree of primitivity have the same covering number.  In Table \ref{tbl:main3}, there are two groups listed as $(A_6 \times A_6).4$, and they are distinguished by saying that one is ``(\#16 in the list)'' and the other is ``(\#18 in the list).''  This is referring to the position in the list of all primitive groups of degree $100$ that is generated by $\GAP$.

\begin{center}
\begin{table}[H] \centering
\begin{tabular}{c|c|c|c}
Group & Degree & Covering Number & Reference\\
\hline
$\PGammaL(2,9)$ & 10 & 3 & $C_2 \times C_2$\\
$M_{10}$ & 10 & 46 & Algorithm \hyperref[subsect:LP]{KNS}\\

$\PSL(3,3)$ & 13 & 157 & Algorithm \hyperref[subsect:LP]{KNS}\\
$\AGammaL(2,4)$ & 16 & 4 & $S_3$\\
$2^4:A_5$ & 16 & 10 & Lemma \ref{lem:Amiri}\\
$2^4:S_6$ & 16 & 13 & Lemma \ref{lem:Amiri}\\
$\ASL(2,4):2$ & 16 & 16 & Algorithm \hyperref[subsect:LP]{KNS}\\
$2^4:A_6$, $2^4:S_5$ & 16 & 16 & Lemma \ref{lem:Amiri}\\
$\AGL(4,2)$ & 16 & 31 & Theorem \ref{thm:AGL}\\
$2^4:A_7$ & 16 & 31 & Algorithm \hyperref[subsect:LP]{KNS}\\
$\PSL(2,16).2$ & 17 & 86 & Algorithm \hyperref[subsect:LP]{KNS}\\
$\PGammaL(2,16)$ & 17 & 86 & Algorithm \hyperref[subsect:LP]{KNS}\\
$\PGammaL(3,4)$ & 21 & 3 & $C_2 \times C_2$\\
$\PSL(3,4)$ & 21 & 141 & Algorithm \hyperref[subsect:LP]{KNS}\\
$\PSigmaL(3,4)$ & 21 & 141 & Algorithm \hyperref[subsect:LP]{KNS}\\
$\PGL(3,4)$ & 21 & 981 & Algorithm \hyperref[subsect:LP]{KNS}\\
$M_{22}:2$ & 22 & 331 & Algorithm \ref{alg}\\
$(A_5 \times A_5):2^2$, $S_5 \Wr 2$ & 25 & 3 & $C_2 \times C_2$\\
$\ASL(2,5):2$ & 25 & 10 & Lemma \ref{lem:Amiri}\\
$\PGammaL(2,25)$ & 26 & 3 & $C_2 \times C_2$\\
$\PSigmaL(2,25)$ & 26 & 127 & Algorithm \hyperref[subsect:LP]{KNS}\\
$\PSL(2,25).2_3$ & 26 & 326 & Algorithm \hyperref[subsect:LP]{KNS}\\
$\ASL(3,3)$, $\AGL(3,3)$ & 27 & 40 & Theorem \ref{thm:AGL}\\
$\PSp(4,3):2$ & 27 & 64 & Algorithm \hyperref[subsect:LP]{KNS}\\
$\PSp(4,3)$ & 27 & 67 & Algorithm \hyperref[subsect:LP]{KNS}\\ 
$\PSigmaL(2,27)$ & 28 & $\ge 167$ & Algorithm \hyperref[subsect:LP]{KNS}\\
$\PGammaU(3,3)$ & 28 & 60 & Algorithm \hyperref[subsect:LP]{KNS}\\
$\PSU(3,3)$ & 28 & 64 & Algorithm \hyperref[subsect:LP]{KNS}\\
$\Sp(6,2)$ & 28 & 64 & Algorithm \ref{alg}\\
$\PGammaL(2,27)$ & 28 & 353 & Algorithm \hyperref[subsect:LP]{KNS}\\
$\PSL(5,2)$ & 31 & 64698 & Algorithm \ref{alg}\\
$\PSL(3,5)$ & 31 & 4031 & Algorithm \hyperref[subsect:LP]{KNS}\\
$\AGL(5,2)$ & 32 & 63 & Theorem \ref{thm:AGL}\\
$\PGammaL(2,32)$ & 33 & 497 & Algorithm \hyperref[subsect:LP]{KNS}\\
$(A_6 \times A_6):2^2$, $S_6 \Wr 2$ & 36 & 3 & $C_2 \times C_2$\\
$(A_6 \times A_6):4$ & 36 & 73 & Algorithm \hyperref[subsect:LP]{KNS}\\ 
$A_6 \Wr 2$ & 36 & 137 & Algorithm \hyperref[subsect:LP]{KNS}\\
$\PSL(4,3)$, $\PGL(4,3)$ & 40 & 2146 & Algorithm \ref{alg}\\
$(A_7 \times A_7):2^2$, $S_7 \Wr 2$ & 49 & 3 & $C_2 \times C_2$\\
$\ASL(2,7):2$, $\ASL(2,7):3$ & 49 & 15 & Lemma \ref{lem:Amiri}\\
$\PSL(3,2) \Wr 2$ & 49 & 114 & Algorithm \hyperref[subsect:LP]{KNS}\\
$(A_7 \times A_7):4$ & 49 & 1716 & Algorithm \ref{alg}\\
$A_7 \Wr 2$ & 49 & $\ge 447$ & Proposition \ref{prop:sigmaleqgeq} (i) \\
$\PGammaL(2,49)$ & 50 & 3 & $C_2 \times C_2$\\
$\PSU(3,5)$, $\PSU(3,5):2$ & 50 & 176 & Algorithm \ref{alg}\\
$\PSigmaL(2,49)$ & 50 & 226 & Algorithm \hyperref[subsect:LP]{KNS}\\
$\PSL(2,49).2_3$ & 50 & 1226 & Algorithm \hyperref[subsect:LP]{KNS}\\
$\PSL(3,3).2$ & 52 & 170 & Algorithm \hyperref[subsect:LP]{KNS}\\
$\PSL(3,4).2^2$ & 56 & 3 & $C_2 \times C_2$\\
$\PSL(3,4).2_1$ & 56 & 162 & Algorithm \hyperref[subsect:LP]{KNS}\\
$\PSL(3,4).2_2$ & 56 & $\ge 138$ & Algorithm \hyperref[subsect:LP]{KNS}\\
\end{tabular}
\caption{Covering numbers of various nonsolvable primitive groups of degree at most $56$}
\label{tbl:main1}
\end{table}
\end{center}

\begin{center}
\begin{table}[H] \centering
\begin{tabular}{c|c|c|c}
Group & Degree & Covering Number & Reference\\
\hline
$\PSL(3,4).2_3$ & 56 & 141 & Algorithm \hyperref[subsect:LP]{KNS}\\
$\PSL(3,7)$, $\PGL(3,7)$ & 57 & 32985 & Algorithm \ref{alg}\\
$\PSL(6,2)$ & 63 & $\ge 56313$ & Algorithm \ref{alg}\\
$2^6:(S_3 \times \GL(3,2))$, $(A_8 \times A_8).2^2$& 64 & 3 & $C_2 \times C_2$\\
$S_8 \Wr 2$, $(\PSL(2,7) \times \PSL(2,7)).2^2$ & 64 & 3 & $C_2 \times C_2$ \\ 
 $\PGL(2,7) \Wr 2$ & 64 & 3 & $C_2 \times C_2$\\
$\AGammaL(3,4)$ & 64 & 4 &    $S_3$\\
$2^6:(S_3 \times \GL(3,2))$ & $64$ & $4$ & $S_3$\\
$\AGammaL(2,8)$ & 64 & 8 &    $C_7:C_3$\\
$2^6:(3.S_6)$ & 64 & 13 & Lemma \ref{lem:Amiri}\\
$2^6:(3 \times \GL(3,2))$ & 64 & 15 & Lemma \ref{lem:Amiri}\\
$2^6:(3.A_6)$ & 64 & 16 & Lemma \ref{lem:Amiri}\\
$\ASigmaL(2,8)$ & 64 & 29 & Lemma \ref{lem:Amiri}\\
$2^6:PGL(2,7)$ & 64 & 15 & Lemma \ref{lem:Amiri}\\
$2^6:A_7$ & 64 & 31 & Lemma \ref{lem:Amiri}\\
$2^6:\SigmaU(3,3)$ & 64 & 60 & Lemma \ref{lem:Amiri}\\
$2^6:\Sp(6,2)$ & 64 & 64 & Lemma \ref{lem:Amiri}\\
$2^6:\GO^-(6,2)$ & 64 & 64 & Lemma \ref{lem:Amiri}\\
$2^6:S_8$ & 64 & 64 & Lemma \ref{lem:Amiri}\\
$2^6:S_7$ & 64 & 64 & Lemma \ref{lem:Amiri}\\
$2^6:\SU(3,3)$ & 64 & 64 & Lemma \ref{lem:Amiri}\\
$2^6:{\rm O}^-(6,2)$ & 64 & 67 & Proposition \ref{prop:sigmavalue} (i) \\
$2^6:A_8$ & 64 & 71 & Proposition \ref{prop:sigmavalue} (ii) \\
 $\ASigmaL(3,4)$ & 64 & 85 & Algorithm \ref{alg}\\
$\ASL(3,4)$, $\AGL(3,4)$ & 64 & 85 & Theorem \ref{thm:AGL}\\
$2^6:(\GL(3,2) \Wr 2)$ & 64 & 114 & Lemma \ref{lem:Amiri}\\
$\AGL(6,2)$ & 64 & 127 & Theorem \ref{thm:AGL}\\
$(PSL(2,7) \times \PSL(2,7)).4$ & 64 & 498 & Proposition \ref{prop:sigmavalue} (iii) \\
$(A_8 \times A_8).4$ & 64 & 2074 & Algorithm \ref{alg}\\
$A_8 \Wr 2$ & 64 & 3426 & Algorithm \ref{alg}\\
$\PSL(2,7) \Wr 2$ & $64$ & $114$ & Algorithm \hyperref[subsect:LP]{KNS}\\
$\PSU(3,4).2$, $\PGammaU(3,4)$ & 65 & 274 & Algorithm \ref{alg}\\
$\PSL(2,64).2$, $\PGammaL(2,64)$ & 65 & 586 & Algorithm \ref{alg}\\
$\Sz(8):3$ & 65 & 1457 & Algorithm \ref{alg}\\
$\PSU(3,4)$ & 65 & 1745 & Algorithm \hyperref[subsect:LP]{KNS}\\
$\PSL(2,64).3$ & 65 & 2080 & Proposition \ref{prop:sigmavalue} (iv) \\
$\PSL(2,16).2$ & 68 & 86 & Algorithm \hyperref[subsect:LP]{KNS}\\
$\PGammaL(3,8)$ & 73 & 56138 & Algorithm \ref{alg}\\
$\PSL(3,8)$ & 73 & 75337 & Algorithm \ref{alg}\\
$3^4:\SL(2,9):2^2$, $3^4:2.A_6:D_8$ & 81 & 3 &    $C_2 \times C_2$\\
$3^4:2.A_6:Q_8$, $\AGammaL(2,9)$ & 81 & 3 &    $C_2 \times C_2$\\
$3^4:4.S_5$, $3^4:4.S_5$, $3^4:8.S_5$ & 81 & 3 &    $C_2 \times C_2$\\
$3^4:(2 \times S_6)$, $3^4:(2 \times A_6.2)$ & 81 & 3 &    $C_2 \times C_2$\\
$3^4:2.\PGammaL(2,9)$, $3^4:(2 \times S_5)$ & 81 & 3 &    $C_2 \times C_2$\\
$(A_9 \times A_9).2^2$, $S_9 \Wr 2$ & 81 & 3 &    $C_2 \times C_2$\\
$(\PSL(2,8) \times \PSL(2,8)).S_3$, $\PSigmaL(2,8) \Wr 2$ & 81 & 4 &    $S_3$\\
$3^4:2.A_5$, $3^4:4.A_5$, $3^4:8.A_5$ & 81 & 10 & Lemma \ref{lem:Amiri}\\
$3^4:2^{1+4}.A_5$, $3^4:A_5$, $3^4:2.A_5$ & 81 & 10 & Lemma \ref{lem:Amiri}\\
\end{tabular}
\caption{Covering numbers of various nonsolvable primitive groups of degree 56 to 81}
\label{tbl:main2}
\end{table}
\end{center}

\begin{center}
\begin{table}[H] \centering
\begin{tabular}{c|c|c|c}
Group & Degree & Covering Number & Reference\\
\hline
$3^4:2.S_6$, $3^4:S_6$ & 81 & 13 & Lemma \ref{lem:Amiri}\\
$3^4:2.A_6$, $3^4:4.A_6$, $3^4:8.A_6$ & 81 & 16 & Lemma \ref{lem:Amiri}\\
$3^4:2.A_5:2$, $3^4:2.A_5:2$, $3^4:2^{1+4}.A_5:2$ & 81 & 16 & Lemma \ref{lem:Amiri}\\
$3^4:A_6$, $3^4:(2 \times A_6)$, $3^4:S_5$, $3^4:S_5$ & 81 & 16 & Lemma \ref{lem:Amiri}\\
$3^4:4.A_6.2$, $3^4:A_6.2$, $3^4:2.\PGL(2,9)$ & 81 & 46 & Lemma \ref{lem:Amiri}\\
$3^4:\Sp(4,3):2$ & 81 & 64 & Lemma \ref{lem:Amiri}\\
$3^4:\Sp(4,3)$& 81 & 67 & Lemma \ref{lem:Amiri}\\
$\ASL(4,3)$, $\AGL(4,3)$ & 81 & 121 & Theorem \ref{thm:AGL}\\
$\PSL(2,8) \Wr 2$ & 81 & 586 & Proposition \ref{prop:sigmavalue} (v) \\
$(\PSL(2,8) \times \PSL(2,8)).6$ & 81 & 586 & Algorithm \ref{alg}\\
$(A_9 \times A_9).4$ & 81 & 24310 & Algorithm \ref{alg}\\
$A_9 \Wr 2$ & 81 & $\ge 10978$ & Algorithm \ref{alg}\\
$\PSL(2,81).2^2$, $\PGammaL(2,81)$ & 82 & 3 &    $C_2 \times C_2$\\
$\PSL(2,81).4$ & 82 & 452 & Algorithm \hyperref[subsect:LP]{KNS}\\
$\PSL(2,81).4$ & 82 & 3322 & Algorithm \ref{alg}\\
$\PSL(2,81).2$ & 82 & $\ge 621$ & Algorithm \hyperref[subsect:LP]{KNS}\\
$\PSp(4,4)$ & 85 & 256 & Algorithm \ref{alg}\\
$\PSL(4,4)$ & 85 & 24277 & Algorithm \ref{alg}\\
$\PSigmaL(4,4)$ & 85 & 45778 & Algorithm \ref{alg}\\
$\PSp(4,4).2$ & 85 & $\ge 196$ & Proposition \ref{prop:sigmaleqgeq} (ii) \\
$\PGammaL(3,9)$ & 91 & 7652 & Algorithm \ref{alg}\\
$\PSL(3,9)$ & 91 & 155611 & Algorithm \ref{alg}\\
$(A_{10} \times A_{10}).2^2$, $S_{10} \Wr 2$& 100 & 3 &    $C_2 \times C_2$\\
$(A_6 \times A_6).2^2$ (all such groups) & 100 & 3 &    $C_2 \times C_2$\\
$(A_6 \times A_6).D_8$ (all such groups) & 100 & 3 &    $C_2 \times C_2$\\
$(A_6 \times A_6).2^3$ & 100 & 3 &    $C_2 \times C_2$\\
$(A_6 \times A_6).(2 \times 4)$ (all such groups) & 100 & 3 &    $C_2 \times C_2$\\
$(A_6 \times A_6).(2 \times D_8)$ (all such groups) & 100 & 3 &    $C_2 \times C_2$\\
$\PGammaL(2,9) \Wr 2$ & 100 & 3 &    $C_2 \times C_2$\\
$(A_6 \times A_6).4$ ($\#16$ in list) & 100 & 1387 & Proposition \ref{prop:sigmavalue} (vi) \\
$(A_6 \times A_6).4$ ($\#18$ in list) & 100 & 2026 & Algorithm \ref{alg}\\
$J_2.2$ & 100 & 2921 & Algorithm \ref{alg}\\
$A_{10} \Wr 2$ & 100 & 30377 & Algorithm \ref{alg}\\
$J_2$ & 100 & $\ge 1063$ & Algorithm \hyperref[subsect:LP]{KNS}\\
$HS:2$ & 100 & $\ge 11859$ & Proposition \ref{prop:sigmaleqgeq} (iii) \\
$(A_{10} \times A_{10}).4$ & 100 & $\ge 22746$ & Proposition \ref{prop:sigmaleqgeq} (iv) \\
$\PSL(3,4).2^2$, $\PSL(3,4).D_{12}$ & 105 & 3 &    $C_2 \times C_2$\\
$\PSL(3,4).S_3$ & 105 & 4 &    $S_3$\\
$\PSL(3,4).6$ & 105 & 386 & Algorithm \ref{alg}\\
$\PSU(4,3).2^2$ (all such groups) & 112 & 3 &    $C_2 \times C_2$\\
$\PSU(4,3).D_8$ & 112 & 3 &    $C_2 \times C_2$\\
$\PSU(4,3)$ & 112 & $\ge 344$ & Proposition \ref{prop:sigmaleqgeq} (v) \\
$\PSU(4,3).2_1$ & 112 & $\ge 256$ & Proposition \ref{prop:sigmaleqgeq} (vi)  \\
$\PSU(4,3).2_2$ & 112 & $\ge 239$ & Proposition \ref{prop:sigmaleqgeq} (vii) \\
$\PSU(4,3).2_3$ & 112 & $ \ge 412$ & Proposition \ref{prop:sigmaleqgeq} (viii) \\
$\PSU(4, 3).4$ & 112 & $\ge 540$ & Algorithm \ref{alg}\\
$\PSL(3,3).2$ & 117 & 170 & Algorithm \hyperref[subsect:LP]{KNS}\\
\end{tabular}
\caption{Covering numbers of various nonsolvable primitive groups of degree 81 to 117}
\label{tbl:main3}
\end{table}
\end{center}

\begin{center}
\begin{table}[H] \centering
\begin{tabular}{c|c|c|c}
Group & Degree & Covering Number & Reference\\
\hline
$\PSL(4,3)$ & 117 & 2146 & Algorithm \ref{alg}\\
$\PSL(4,3).2$ & 117 & $\ge 242$ & Proposition \ref{prop:sigmaleqgeq} (ix) \\
${\rm PSO}^-(8,2)$ & 119 & 256 & Algorithm \ref{alg}\\
${\rm O}^-(8,2)$ & 119 & $\ge 25706$ & Proposition \ref{prop:sigmaleqgeq} (x) \\
$\PSL(3,4).2^2$ & 120 & 3 &    $C_2 \times C_2$\\
$\Sp(8,2)$ & 120 & 256 & Proposition \ref{prop:sigmavalue} (vii) \\ 
${\rm PSO}^+(8,2)$ & 120 & 256 & Algorithm \ref{alg}\\
${\rm O}^+(8,2)$ & 120 & $\ge 204$ & Proposition \ref{prop:sigmaleqgeq} (xi) \\
$(A_{11} \times A_{11}).2^2$, $S_{11} \Wr 2$ & 121 & 3 &    $C_2 \times C_2$\\
$11^2:(2.A_5)$, $11^2:(5 \times 2.A_5)$ & 121 & 10 & Lemma \ref{lem:Amiri}\\
$11^2:(\SL(2,11):2)$, $11^2:(5 \times \SL(2,11))$ & 121 & 67 & Lemma \ref{lem:Amiri}\\
$M_{11} \Wr 2$ & 121 & 266 & Algorithm \ref{alg}\\
$A_{11} \Wr 2$ & 121 & 6380772 & Algorithm \ref{alg}\\
$\PSL(2,11) \Wr 2$ & 121 & $\ge 570$ & Proposition \ref{prop:sigmaleqgeq} (xii) \\
$\PSL(5,3)$ & 121 & $\ge 393030144$ & Proposition \ref{prop:sigmageq} (i) \\
$(A_{11} \times A_{11}).4$ & 121 & $\ge 213444$ & Proposition \ref{prop:sigmageq} (ii) \\
$\PGammaL(2,121)$ & 122 & 3 &    $C_2 \times C_2$\\
$\PSL(2,121).2_3$ & 122 & 7382 & Algorithm \ref{alg}\\
$\PSigmaL(2,121)$ & 122 & $\ge 671$ & Algorithm \ref{alg}\\
$A_5 \Wr S_3$, $A_5^3.S_3$, $(A_5 \times A_5 \times A_5).S_4$ & 125 & 4 &    $S_3$\\
$(A_5 \times A_5 \times A_5).2^2.3$, $S_5 \Wr 3$ & 125 & 5 &    $A_4$\\
$5^3:A_5$, $5^3:(2 \times A_5)$, $5^3:(4 \times A_5)$ & 125 & 10 & Lemma \ref{lem:Amiri}\\
$5^3:S_5$, $5^3:(2.S_5)$ & 125 & 16 & Lemma \ref{lem:Amiri}\\
$\ASL(3,5)$, $\AGL(3,5)$ & 125 & 156 & Theorem \ref{thm:AGL}\\
$5^3:(\SL(3,5):2)$ & 125 & 156 & Algorithm \ref{alg}\\
$A_5 \Wr 3$ & 125 & $\ge 216$ & Algorithm \ref{alg}\\
$(A_5 \times A_5 \times A_5).6$ & 125 & $\ge 1000$ & Algorithm \ref{alg}\\
$\PGammaU(3,5)$ & 126 & 4 &    $S_3$\\
$\PSigmaL(2,125)$ & 126 & 7876 & Algorithm \ref{alg}\\
$\PGU(3,5)$ & 126 & $\ge 6000$ & Algorithm \ref{alg}\\
$\PGammaL(2,125)$ & 126 & $\ge 7750$ & Algorithm \ref{alg}\\
$\PSL(7,2)$ & 127 & $\ge 184308203520$ & Proposition \ref{prop:sigmageq} (iii) \\
$\AGL(7,2)$ & 128 & 255 & Theorem \ref{thm:AGL}\\
$\PGammaL(2,128)$ & 129 & 8129 & Algorithm \ref{alg}\\ 

\end{tabular}
\caption{Covering numbers of various nonsolvable primitive groups of degree 117 to 129}
\label{tbl:main4}
\end{table}
\end{center}

We conclude this section with a table that lists the primitive monolithic groups $G$ such that $G$ has a degree of primitivity at most $129$ and $\sigma(G)$ has not yet been determined exactly, and, for each such group $G$, Table \ref{tbl:main5} lists the best known bounds on $\sigma(G)$. 

\begin{center}
\begin{table}[H] \centering
\begin{tabular}{c|c|c|c|c}
Group & Degree & Lower bound & Upper bound & Reference\\
\hline
$\PSigmaL(2,27)$ & 28 & 167 & 184 & Algorithm \hyperref[subsect:LP]{KNS}\\
$A_7 \Wr 2$ & 49 & 447 & 667 & Proposition \ref{prop:sigmaleqgeq} (i) \\
$\PSL(3,4).2_2$ & 56 & 138 & 166 & Algorithm \hyperref[subsect:LP]{KNS}\\
$\PSL(6,2)$ & 63 & 56313 & 57010 & Algorithm \ref{alg}\\
$A_9 \Wr 2$ & 81 & 10978 & 30178 & Algorithm \ref{alg}\\
$\PSL(2,81).2$ & 82 & 621 & 731 & Algorithm \hyperref[subsect:LP]{KNS}\\
$\PSp(4,4).2$ & 85 & 196 & 222 & Proposition \ref{prop:sigmaleqgeq} (ii) \\
$J_2$ & 100 & 1063 & 1121 & Algorithm \hyperref[subsect:LP]{KNS}\\
$HS:2$ & 100 & 11859 & 22375 & Proposition \ref{prop:sigmaleqgeq} (iii) \\
$(A_{10} \times A_{10}).4$ & 100 & 22746 & 30377 & Proposition \ref{prop:sigmaleqgeq} (iv) \\
$\PSU(4,3)$ & 112 & 344 & 442 & Proposition \ref{prop:sigmaleqgeq} (v) \\
$\PSU(4,3).2_1$ & 112 & 256 & 554 & Proposition \ref{prop:sigmaleqgeq} (vi) \\
$\PSU(4,3).2_2$ & 112 & 239 & 365 & Proposition \ref{prop:sigmaleqgeq} (vii) \\
$\PSU(4,3).2_3$ & 112 & 412 & 554 & Proposition \ref{prop:sigmaleqgeq} (viii) \\
$\PSU(4, 3).4$ & 112 & $540$ & $652$ & Algorithm \ref{alg}\\
$\PSL(4,3).2$ & 117 & 242 & 365 & Proposition \ref{prop:sigmaleqgeq} (ix) \\
${\rm O}^-(8,2)$ & 119 & 25706 & 26283 & Proposition \ref{prop:sigmaleqgeq} (x) \\
${\rm O}^+(8,2)$ & 120 & 204 & 765 & Proposition \ref{prop:sigmaleqgeq} (xi) \\
$\PSL(2,11) \Wr 2$ & 121 & 570 & 926 & Proposition \ref{prop:sigmaleqgeq} (xii) \\
$\PSL(5,3)$ & 121 & 393030144 & & Proposition \ref{prop:sigmageq} (i) \\
$(A_{11} \times A_{11}).4$ & 121 & 213444 & & Proposition \ref{prop:sigmageq} (ii) \\
$\PSigmaL(2,121)$ & 122 & 671 & 794 & Algorithm \ref{alg}\\
$A_5 \Wr 3$ & 125 & 216 & 342 & Algorithm \ref{alg}\\
$(A_5 \times A_5 \times A_5).6$ & 125 & 1000 & 1217 & Algorithm \ref{alg}\\
$\PGU(3,5)$ & 126 & 6000 & 6526 & Algorithm \ref{alg}\\
$\PGammaL(2,125)$ & 126 & 7750 & 7876 & Algorithm \ref{alg}\\
$\PSL(7,2)$ & 127 & 184308203520 & & Proposition \ref{prop:sigmageq} (iii) \\

\end{tabular}
\caption{Bounds on covering numbers of various nonsolvable primitive groups}
\label{tbl:main5}
\end{table}
\end{center}

 
\appendix
\section{Calculations for specific groups}
\label{sect:calcspecific}

We use the following notation in the tables.  The notation $\calM_i$ indicates a conjugacy class of maximal subgroups.  Below the symbol $\calM_i$, the number in parentheses indicates the number of conjugate subgroups in the class.  The notation ``$cl_{m,j}$'' refers to a class of elements of order $m$; the ``$j$'' will be omitted when we are considering a single class of this order.  If the $(cl_m, \calM_i)$-entry of the table is $n_k$, then each subgroup of $\calM_i$ contains $n$ elements of class $cl_m$, and each element in $cl_m$ is contained in $k$ subgroups of $\calM_i$.  Instead of writing $n_1$, we will write $n, P$ to indicate that the elements of $cl_m$ are partitioned among the subgroups in $\calM_i$.  If the $(cl_m$, $\calM_i)$-entry is written as $n$, then the information about how many subgroups in $\calM_i$ contain a given element of $cl_m$ is unimportant to the proof and is omitted.   We observe that the smallest primitivity degree of each of the following subgroups is an index of one of its maximal subgroups, and hence this value appears as an index in the corresponding table (when such a table is provided).  

One argument that is used repeatedly in the following propositions is the following, which we state here for emphasis: if there are $c$ elements from class $cl_j$ remaining to be covered and the $(cl_j,\calM_i)$-entry of the table is $n_k$, then at least $\lceil c/n \rceil$ subgroups from $\calM_i$ are needed to cover the $c$ elements of class $cl_j$.  In particular, if a class of maximal subgroups $\calM_i$ has size $m$ and the $(cl_j, \calM_i)$-entry of the table is $n_k$, then at least $m/k$ subgroups from $\calM_i$ are needed to cover the elements of class $cl_j$.

\begin{prop}
\label{prop:sigmavalue}
We have the following covering number values:
\begin{enumerate}[(i)]
\item $\sigma(2^6:{\rm O}^-(6,2)) = 67$;
\item $\sigma(2^6:A_8) = 71$;
\item $\sigma((\PSL(2,7) \times \PSL(2,7)).4) = 498$;
\item $\sigma(\PSL(2,64).3) = 2080$;
\item $\sigma(\PSL(2,8) \Wr 2) = 586$;
\item $\sigma((A_6 \times A_6).4) = 1387$;
\item $\sigma(\Sp(8,2)) = 256$.
\end{enumerate}
\end{prop}

\begin{proof}
(i) Using Algorithm \ref{alg}, we have that $\sigma({\rm O}^-(6,2)) = 67$, and so $\sigma(2^6:{\rm O}^-(6,2)) \le 67$. Suppose, for the purpose of contradiction, that $\sigma(2^6:{\rm O}^-(6,2)) < 67$, and let $\mathcal{B}$ be this smaller cover.  By Lemma \ref{lem:Garonzi}, this means that all $64$ conjugates of the point stabilizer in the primitive action on $64$ points are contained in $\mathcal{B}$.  By $\GAP$ \cite{GAP}, there is a class of elements of order $12$ that are not contained in the point stabilizers.  The most number of elements of this class that are contained in a maximal subgroup is $1920$, and so at least an additional $18$ subgroups are needed to cover this class.  However, $64 + 18 > 67$, a contradiction to minimality.  Therefore, the covering number of $2^6:{\rm O}^-(6,2)$ is $67$.

(ii) Since $\sigma(A_8) = 71$, if $\sigma(2^6:A_8) < 71$, then by Lemma \ref{lem:Garonzi}, any minimal cover of $2^6:A_8$ would have to contain all maximal subgroups isomorphic to $A_8$.  However, by $\GAP$, there are a total of $128$ maximal subgroups isomorphic to $A_8$, a contradiction.

(iii) By $\GAP$, there are four classes of maximal subgroups, and we have the following distribution of elements:

\begin{center}
\begin{table}[H]\centering
\begin{tabular}{c|c|c|c|c}
 { }        & $\calM_1$ & $\calM_2$ & $\calM_3$ & $\calM_4$  \\ 
 { }        &  (1)      & (64)   & (441)   & (784)   \\
 \hline
 $cl_{24}$ & $4704, P$ & 0 & 0 & 0 \\
 $cl_{16}$ & 0 & 0 & $16,P$ & 0   \\
 $cl_{12,1}$ & 0 & $294_2$ & 0 & $12,P$ \\
  $cl_{12,2}$ & 0 & $294_2$ & 0 & $12,P$ \\

 \end{tabular}
 \caption{Element distribution in $\PSL(2,7)^2.4$}
 \label{tbl:PSL274}
 \end{table}
 \end{center}

The unique minimal normal subgroup is the only class containing elements from $cl_{24}$.  Moreover, the elements of $cl_{16}$ are partitioned among the $441$ subgroups in $\calM_3$, so these $441$ subgroups are also contained in a minimal cover.  Only the two classes $cl_{12,1}$ and $cl_{12,2}$ are left uncovered after including these $442$ subgroups.  Using Algorithm \hyperref[subsect:LP]{KNS} and GUROBI \cite{Gu} for the elements in these two classes, we find that the minimal cover of these two classes contains $56$ subgroups.  Therefore, the covering number of $(\PSL(2,7) \times \PSL(2,7)).4$ is $498$.

(iv) By $\GAP$, we have the following distribution of elements:

\begin{center}
\begin{table}[H]\centering
\begin{tabular}{c|c|c|c|c|c|c}
 { }        & $\calM_1$ & $\calM_2$ & $\calM_3$ & $\calM_4$ & $\calM_5$ & $\calM_6$ \\ 
 { }        &  (1)      & (65)   & (520)   & (2016) & (2080) & (4368)   \\
 \hline
 $cl_{63}$ & $12480, P$ & 384 & 0 & 0 & 6 & 0 \\
 $cl_{15}$ & 0 & 0 & 0 & $26,P$ & 0 & 12  \\
 $cl_{9,1}$ & 0 & $2688_2$ & 168 & 0 & 42 & 0 \\
 $cl_{9,2}$ & 0 & $2688_2$ & 168 & 0 & 42 & 0 \\

 \end{tabular}
 \caption{Element distribution in $\PSL(2,64).3$}
 \label{tbl:PSL2643}
 \end{table}
 \end{center}

The classes $cl_{15}$, $cl_{9,1}$, and $cl_{9,2}$ are not contained in the minimal normal subgroup in $\calM_1$.  The elements of $cl_{15}$ are partitioned among the 2016 subgroups of $\calM_4$, and no subgroup contains more elements of $cl_{15}$ than a subgroup in $\calM_4$ does.  Using Algorithm \hyperref[subsect:LP]{KNS} and GUROBI, the minimal cover of $cl_{9,1}$ and $cl_{9,2}$ has size $64$, and calculations in GAP show that a random choice of $64$ subgroups from $\calM_2$ (say, the first $64$ in a given list, excluding the last) plus the $2016$ aforementioned maximal subgroups from $\calM_4$ are a cover.  Therefore, the covering number of $\PSL(2,64).3$ is $2080$.

(v) By $\GAP$, we have the following distribution of elements in $\PSL(2,8) \Wr 2$:

\begin{center}
\begin{table}[H]\centering
\begin{tabular}{c|c|c|c|c|c}
 { }        & $\calM_1$ & $\calM_2$ & $\calM_3$ & $\calM_4$ & $\calM_5$  \\ 
 { }        &  (1)      & (504)   & (81)   & (784) & (1296)   \\
 \hline
 $cl_{63}$ & $8064, P$ & 0 & 0 & 0 & 0 \\
 $cl_{18}$ & 0 & $56,P$ & 0 & $36$ & 0   \\
 $cl_{4}$ & 0 & 0 & $392,P$ & 162 & 98  \\

 \end{tabular}
 \caption{Element distribution in $\PSL(2,8) \Wr 2$}
 \label{tbl:PSL28wr2}
 \end{table}
 \end{center}

By Algorithm \hyperref[subsect:LP]{KNS} and GUROBI, the subgroups from $\calM_2$ and $\calM_3$ are a minimal cover of $cl_{18}$ and $cl_4$.  However, these two classes together are not a cover, whereas including the minimal normal subgroup from $\calM_1$ with these is a cover.  Therefore, the covering number of $\PSL(2,8) \Wr 2$ is $586$.

(vi) Here $(A_6 \times A_6).4$ is number $16$ of the list AllPrimitiveGroups(NrMovedPoints,100) returned by $\GAP$. A class of elements of order $40$ is only contained in the minimal normal subgroup, and a class of elements of order $20$ is only contained in the $1296$ subgroups from another class.  The only elements not covered by this class are four classes of elements of order $16$.  By Algorithm \hyperref[subsect:LP]{KNS} and GUROBI, a minimal cover of these elements contains $90$ subgroups.  The result follows.

(vii) Using Algorithm \ref{alg}, $\sigma({\rm O}^+(8,2):2) = \sigma({\rm O}^-(8,2):2) = 256$. By Lemma \ref{lem:Garonzi}, if $\sigma(\Sp(8,2)) < 256$, then all maximal subgroups isomorphic to either ${\rm O}^+(8,2):2$ or ${\rm O}^-(8,2):2$ are in such a minimal cover.  However, there are $136 + 120 = 256$ such subgroups, so $\sigma(\Sp(8,2)) \ge 256$.  On the other hand, calculations in $\GAP$ show that all $256$ subgroups isomorphic to either ${\rm O}^+(8,2):2$ or ${\rm O}^-(8,2):2$ are a cover.  The result follows.
\end{proof}

\begin{prop}
\label{prop:sigmaleqgeq}
We have the following lower and upper bounds for the indicated covering number values:
\begin{enumerate}[(i)]
\item $447 \le \sigma(A_7 \Wr 2) \le 667$;
\item $196 \le \sigma(\PSp(4,4).2) \le 222$;
\item $11859 \le \sigma(HS:2) \le 22375$;
\item $22746 \le \sigma((A_{10} \times A_{10}).4) \le 30377$;
\item $344 \le \sigma(\PSU(4,3)) \le 442$;
\item  $256 \le \sigma(\PSU(4,3).2) \le 554$;
\item $239 \le \sigma(\PSU(4,3).2) \le 365$;
\item $412 \le \sigma(\PSU(4,3).2) \le 554$;
\item $242 \le \sigma(\PSL(4,3).2) \le 365$;
\item $25706 \le \sigma({\rm O}^-(8,2)) \le 26283$;
\item $204 \le {\rm O}^+(8,2) \le 765$;
\item $570 \le \sigma(\PSL(2,11) \Wr 2) \le 926$.
\end{enumerate}
\end{prop}

\begin{proof}
(i) Using $\GAP$, we find the following distribution of elements:

\begin{center}
\begin{table}[H]\centering
\begin{tabular}{c|c|c|c|c|c|c|c|c}
 { }        & $\calM_1$ & $\calM_2$ & $\calM_3$ & $\calM_4$  & $\calM_5$ & $\calM_6$ & $\calM_7$ & $\calM_8$ \\ 
 { }        &  (1)      & (2520)   & (2520)   & (1225)  & (441) & (49) & (225) & (225) \\
 \hline
 $cl_{14}$ & 0 & $360$ & 0 & 0 & 0 & 0 & $4032,P$ & $4032,P$ \\
 $cl_{12}$ & 0 & 0 & $420$ & $432$ & $2400_2$ & 0 & 0 & 0 \\
 $cl_{3}$ & $39200,P$ & 0 & 0 & $320$ & 0 & $3200$ & 0 & 0 \\
 \end{tabular}
 \caption{Element distribution in $A_7 \Wr 2$}
 \label{tbl:A7wr2}
 \end{table}
 \end{center}
Either $\calM_7$ or $\calM_8$ along with the minimal normal subgroup in $\calM_1$ constitute a minimal cover of $cl_{14}$ and $cl_{12}$.  For the lower bound, it takes at least $221$ subgroups from $\calM_5$ to cover the elements $cl_{12}$.  The upper bound comes from Algorithm \ref{alg}.

(ii) First, by $\GAP$, we have the following distribution of elements:

 \begin{center}
\begin{table}[H]\centering
\begin{tabular}{c|c|c|c|c|c|c|c|c}
 { }        & $\calM_1$ & $\calM_2$ & $\calM_3$ & $\calM_4$  & $\calM_5$ & $\calM_6$ & $\calM_7$ & $\calM_8$ \\ 
 { }        &  (1)      & (85)   & (85)   & (120)  & (120) & (136) & (136) & (1360) \\
 \hline
 $cl_{10}$ & 0 & 0 & 0 & 0 & 0 & $1440, P$ & $1440,P$ & 144 \\
 $cl_{8}$ & 0 & $1440, P$ & $1440,P$ & $2040_2$ & $2040_2$ & 0 & 0 & 0 \\

 \end{tabular}
 \caption{Element distribution in $\PSp(4,4).2$}
 \label{tbl:PSp442}
 \end{table}
 \end{center}
The elements of $cl_{10}$ are partioned among the subgroups in $\calM_6$ and $\calM_7$ in each class, and each of these classes contains $136$ subgroups, so at least $136$ subgroups are necessary to cover these elements.  On the other hand, no maximal subgroup containing an element of $cl_{10}$ contains an element from $cl_8$.  The most number of elements from $cl_8$ in a single maximal subgroup is $2040$, and each element of $cl_8$ is contained in exactly two of the $120$ subgroups in each of $\calM_4$ or $\calM_5$.  Hence it takes at least $120/2$ subgroups to cover these elements, giving a lower bound of $136 + 60 = 196$.  On the other hand, using $\GAP$, it can be verified that the minimal normal subgroup in $\calM_1$ together with $\calM_2$ and $\calM_6$ is a cover, giving the upper bound of $222$.

(iii) By $\GAP$, we have the following distribution of elements in $HS:2$:

\begin{center}
\begin{table}[H]\centering
\begin{tabular}{c|c|c|c|c|c}
 { }        & $\calM_1$ & $\calM_2$ & $\calM_3$ & $\calM_4$ & $\calM_5$  \\ 
 { }        &  (1)      & (100)   & (1100)   & (1100) & (3850)   \\
 \hline
 $cl_{11}$ & $8064000, P$ & $80640,P$ & 0 & 0 & 0 \\
 $cl_{30}$ & 0 & 0 & 0 & $2688,P$ & 0   \\
 $cl_{20,1}$ & 0 & 0 & 0 & 0 & 0  \\
 $cl_{20,2}$ & 0 & 0 & 0 & 0 & 0  \\
 $cl_{10}$   & 0 & $88704, P$ & 8064 & 0 & 2304\\
 \end{tabular}
 \caption{Element distribution in $HS:2$}
 \label{tbl:HS2_1}
 \end{table}
 \end{center}

 \begin{center}
\begin{table}[H]\centering
\begin{tabular}{c|c|c|c|c|c}
 { }        & $\calM_6$ & $\calM_7$ & $\calM_8$ & $\calM_9$ & $\calM_{10}$  \\ 
 { }        &  (4125)      & (5775)   & (15400)   & (22176) & (36960)   \\
 \hline
 $cl_{11}$ & 0 & 0 & 0 & 0 & 0 \\
 $cl_{30}$ & 0 & 0 & 0 & 0 & 80   \\
 $cl_{20,1}$ & 0 & $768,P$ & 0 & $200,P$ & 0  \\
 $cl_{20,1}$ & 0 & 0 & $288,P$ & $400_2$ & $120,P$  \\
 $cl_{10}$ & 0 & 0 & 0 & $400,P$ & 0 \\
 \end{tabular}
 \caption{Element distribution in $HS:2$, cont.}
 \label{tbl:HS2_2}
 \end{table}
 \end{center}
Now, using $\GAP$, the subgroups in classes $\calM_2$, $\calM_4$, $\calM_7$, and $\calM_8$ form a cover, giving the upper bound.  On the other hand, the elements of $cl_{30}$ are covered by the $1100$ maximal subgroups in $\calM_4$.  At least $5775$ different subgroups are needed for $cl_{20,1}$, and the minimal normal subgroup in $\calM_1$ is a minimal cover of a class of elements of order $11$.  At this point, at most $2442000$ elements can possibly be covered from $cl_{20,2}$, being $120\cdot 1100 + 5775 \cdot 400 = 2442000$.  Since $15400 \cdot 288 - (120 \cdot 1100 + 5775 \cdot 400) = 1993200$, this leaves at least $1993200$ elements still uncovered.  The most elements of this class in any maximal subgroup is $400$, which means at least an additional $4983$ subgroups are required to cover these elements.  Since $1100 + 5775 + 1 + 4983 = 11859$, the covering number is bounded below by $11859$.

(iv) Using $\GAP$, we obtain the following information about some classes of elements in $(A_{10} \times A_{10}).4$.
 
 \begin{center}
\begin{table}[H]\centering
\begin{tabular}{c|c|c|c|c}
 { }        & $\calM_1$ & $\calM_2$ & $\calM_3$ & $\calM_4$  \\ 
 { }        &  (1)      & (44100)   & (14400)   & (2025)   \\
 \hline
 $cl_{72}$ & $182891520000, P$ & 0 & 0 & 0 \\
 $cl_{20}$ & 0 & 0 & 0 & 0   \\
 $cl_{28}$ & 0 & 0 & $3259200, P$ & $232243200$ \\
 $cl_{24,1}$ & 0 & $2073600$ & $12700800_2$ & 0 \\
 $cl_{24,2}$ & 0 & $2073600$ & $12700800_2$ & 0 \\

 \end{tabular}
 \caption{Element distribution in $(A_{10} \times A_{10}).4$}
 \label{tbl:A10_1}
 \end{table}
 \end{center}

 \begin{center}
\begin{table}[H]\centering
\begin{tabular}{c|c|c|c|c}
 { }        & $\calM_5$ & $\calM_6$ & $\calM_7$ &  $\calM_8$ \\ 
 { }         & (100)     & (893025)  & (15876)   &  (6350400) \\
 \hline
 $cl_{72}$  & 0 & 0 & 0 & 0\\
 $cl_{20}$  & 0 & $737280$ & $41472000, P$ & $103680$  \\
 $cl_{28}$ & $4702924800, P$ & 0 & 0 & 0\\
 $cl_{24,1}$ & 0 & 307200 & 0 & 0\\
 $cl_{24,2}$  & 0 & 307200 & 0 & 0\\

 \end{tabular}
 \caption{Element distribution in $(A_{10} \times A_{10}).4$, cont.}
 \label{tbl:A10_2}
 \end{table}
 \end{center}
Using Algorithm \ref{alg}, the subgroups in classes $\calM_1$, $\calM_3$, $\calM_5$ and $\calM_7$ collectively form a cover, giving the upper bound.  On the other hand, the information in Tables \ref{tbl:A10_1} and \ref{tbl:A10_2} shows the necessity of the subgroup in $\calM_1$ to cover the elements in $cl_{72}$, and it takes at least $15876$ additional subgroups to cover $cl_{20}$.  At this point, since 
\[ 14400 \cdot \frac{12700800}{2} - 15876 \cdot 307200 = 86568652800,\]
at least $86568652800$ elements from each of $cl_{24,1}$ and $cl_{24,2}$ are still uncovered. Because 
\[\frac{86568652800}{12700800} = 6816,\]
at least $6816$ subgroups are still needed to cover the elements from these classes.  Noting that 
\[ \left\lceil \frac{470292480000 - 6816 \cdot 32659200}{4702924800} \right\rceil = 53, \]
at least $53$ more subgroups from $\calM_5$ are needed to cover the elements of $cl_{28}$, giving a lower bound of $22746$.

(v) Using $\GAP$, we have the following information about elements of $\PSU(4,3)$.
 
 \begin{center}
\begin{table}[H]\centering
\begin{tabular}{c|c|c|c|c|c|c|c|c}
 { }        & $\calM_1$ & $\calM_2$ & $\calM_3$ & $\calM_4$  & $\calM_5$ & $\calM_6$ & $\calM_7$ & $\calM_8$ \\ 
 { }        &  (112)      & (126)   & (126)   & (162)  & (162) & (280) & (540) & (567) \\
 \hline
 $cl_{7}$ & 0 & 0 & 0 & $2880,P$ & $2880,P$ & 0 & $864,P$ & 0 \\
 $cl_{9,1}$ & $1080,P$ & $2880_3$ & 0 & 0 & 0 & $432,P$ & 0 & 0   \\
$cl_{9,2}$ & $1080,P$ & 0 &  $2880_3$ & 0 & 0 & $432,P$ & 0 & 0   \\
 $cl_{8}$ & 0 & 0 & 0 & 0 & 0 & $2916_2$ & $1512_2$ & $720,P$ \\

 \end{tabular}
 \caption{Element distribution in $\PSU(4,3)$}
 \label{tbl:PSU43_1}
 \end{table}
 \end{center}
 
  \begin{center}
\begin{table}[H]\centering
\begin{tabular}{c|c|c|c|c|c|c|c|c}
 { }        & $\calM_9$ & $\calM_{10}$ & $\calM_{11}$ & $\calM_{12}$  & $\calM_{13}$ & $\calM_{14}$ & $\calM_{15}$ & $\calM_{16}$ \\ 
 { }        &  (567)      & (1296)   & (1296)   & (1296)  & (1296) & (2835) & (4536) & (4536) \\
 \hline
 $cl_{7}$ & 0 & 360 & 360 & 360 & 360 & 0 & 0 & 0 \\
 $cl_{9,1}$ & 0 & 0 & 0 & 0 & 0 & 0 & 0 & 0   \\
$cl_{9,2}$  & 0 & 0 & 0 & 0 & 0 & 0 & 0 & 0   \\
 $cl_{8}$ & $720,P$ & 0 & 0 & 0 & 0 & $144$ & $180$ & $180$ \\

 \end{tabular}
 \caption{Element distribution in $\PSU(4,3)$, cont.}
 \label{tbl:PSU43_2}
 \end{table}
 \end{center}
First, Algorithm \ref{alg} shows that the $442$ subgroups in $\calM_4$ and $\calM_6$ form a cover.  On the other hand, the information in Tables \ref{tbl:PSU43_1} and \ref{tbl:PSU43_2} shows that at least $162$ subgroups are needed to cover $cl_7$.  Suppose that we use $162 - m$ subgroups from $\calM_4$ and $\calM_5$ and that we use $m_7$ subgroups from $\calM_7$.  This means we use $(162 - m) + m_7$ groups to cover $cl_7$.  This implies that $864m_7 \ge 2880m$, that is, this implies that $m \le 3m_7/10$, and so $162 + (m_7 - m) \ge 162 + 7m_7/10$.  For each group that we use from class $\calM_7$, potentially $1512$ elements from $cl_8$ are covered.  Since

\[ \frac{408240 - 1512m_7}{2916} = 140 - \frac{14m_7}{27},\]
we still need at least $140 - 14m_7/27$ groups to cover $cl_8$.  Noting that 

\[ \left(162 + \frac{7m_7}{10}\right) + \left(140 - \frac{14m_7}{27}\right) = 302 + \frac{49m_7}{270} \ge 302,\]
at least $302$ subgroups are required to cover classes $cl_7$ and $cl_8$.  Since $120960 - 140 \cdot 432 = 60480$, at the very least $60480$ of the elements from each of $cl_{9,1}$ and $cl_{9,2}$ are still uncovered. Because $2 \cdot 60480/2880 = 42$, an additional $42$ subgroups are needed, and hence at least $344$ subgroups are needed to cover $\PSU(4,3)$.

(vi) $\PSU(4,3).2$ is the group $\U(4,3).2_1$ in the ATLAS \cite{ATLAS}. Using $\GAP$, we have the following information about elements of $\PSU(4,3).2$.
 
\begin{center}
\begin{table}[H]\centering
\begin{tabular}{c|c|c|c|c|c|c|c}
 { }        & $\calM_1$ & $\calM_2$ & $\calM_3$ & $\calM_4$  & $\calM_5$ & $\calM_6$ & $\calM_7$ \\ 
 { }        &  (1)      & (112)   & (126)   & (126)  & (162) & (162) & (280) \\
 \hline
 $cl_{14}$ & 0 & 0 & 0 & 0 & $2880,P$ & $2880,P$ & 0  \\
 $cl_{10}$ & 0 & $11664_2$ & $5184,P$ & $5184,P$ & 0 & 0 & 0  \\
$cl_{6}$ & 0 & 0 &  0 & 0 & 0 & 0 & $108,P$ \\
 
 \end{tabular}
 \caption{Element distribution in $\PSU(4,3).2$}
 \label{tbl:PSU432_11}
 \end{table}
 \end{center}
 
 \begin{center}
\begin{table}[H]\centering
\begin{tabular}{c|c|c|c|c|c|c}
 { }        & $\calM_8$ & $\calM_9$ & $\calM_{10}$ & $\calM_{11}$ & $\calM_{12}$ & $\calM_{13}$ \\ 
 { }       & (540) & (567) & (567) & (2835) & (4536) & (4536) \\
 \hline
 $cl_{14}$ & $864,P$ & 0 & 0 & 0 & 0 & 0 \\
 $cl_{10}$  & 0 & 0 & 0 & 0 & 144 & 144  \\
$cl_{6}$ & $560_{10}$ & 0 & 0 & $96_9$ & 0 & 0  \\
 
 \end{tabular}
 \caption{Element distribution in $\PSU(4,3).2$, cont.}
 \label{tbl:PSU432_12}
 \end{table}
 \end{center}
First, using $\GAP$, the subgroups in $\calM_2$, $\calM_5$, $\calM_7$ are a cover, giving the upper bound of $554$.  On the other hand, examining Tables \ref{tbl:PSU432_11} and \ref{tbl:PSU432_12}, we see that at least $54$ subgroups are needed to cover $cl_6$.  Supposing that $54$ subgroups from $\calM_8$ are used to cover $cl_6$, which would be optimal, we would then have covered $864\cdot54$ elements from $cl_{14}$. Since 
\[ \lceil 864(540 - 54)/2880\rceil = 146,\]
at least $146$ subgroups are still needed to cover $cl_{14}$.  Finally, at least $56$ subgroups are still needed to cover $cl_{10}$, since no subgroup that contains an element of $cl_{14}$ or $cl_{6}$ contains an element of $cl_{10}$.  Therefore, at least $256$ subgroups are needed to cover $\PSU(4,3).2$. 

(vii) $\PSU(4,3).2$ is the group $\U(4,3).2_2$ in the ATLAS \cite{ATLAS}. Using $\GAP$, we have the following information about elements of $\PSU(4,3).2$.
 
 \begin{center}
\begin{table}[H]\centering
\begin{tabular}{c|c|c|c|c|c|c}
 { }        & $\calM_1$ & $\calM_2$ & $\calM_3$ & $\calM_4$  & $\calM_5$ & $\calM_6$ \\ 
 { }        &  (1)      & (112)   & (126)   & (126)  & (280) & (540)\\
 \hline
 $cl_{10}$ & 0 & 0 & $5184,P$ & $5184,P$ & 0 & 0   \\
 $cl_{18}$ & 0 & $3240,P$ & 0 & $2880$ & 1296 & 0  \\
$cl_{12,1}$ & 0 & $4860,P$ &  4320 & 0 & 0 & 0  \\
$cl_{12,2}$ & 0 & 0 & 0 & $4320_2$ & 972 & 2016  \\
$cl_{8}$ & 0 & 0 & $6480_2$  & 0 & 2916 & 1512  \\
$cl_{7}$ & $933120,P$ & 0 &  0 & 0 & 0 & 1728  \\
 \end{tabular}
 \caption{Element distribution in $\PSU(4,3).2$}
 \label{tbl:PSU432_21}
 \end{table}
 \end{center}
 
 \begin{center}
\begin{table}[H]\centering
\begin{tabular}{c|c|c|c|c|c}
 { }        & $\calM_7$ & $\calM_8$ & $\calM_9$ & $\calM_{10}$  & $\calM_{11}$ \\ 
 { }        &  (567)      & (567)   & (1296)   & (1296)  & (2835)\\
 \hline
 $cl_{10}$ & 0 & 2304 & 504 & 504 & 0   \\
 $cl_{18}$ & 0 & 0 & 0 & 0 & 0  \\
$cl_{12,1}$ & 960 & 0 &  420 & 420 & 192  \\
$cl_{12,2}$ & 0 & 0 & 0 & 0 & 96  \\
$cl_{8}$ & 720 & 720 & 0  & 0 & 144 \\
$cl_{7}$ & 0 & 0 &  720 & 720 & 0  \\
 \end{tabular}
 \caption{Element distribution in $\PSU(4,3).2$, cont.}
 \label{tbl:PSU432_22}
 \end{table}
 \end{center}
Using $\GAP$, we see that the subgroups in $\calM_1$, $\calM_2$, $\calM_3$, and $\calM_4$ constitute a cover, demonstrating the upper bound.  On the other hand, Tables \ref{tbl:PSU432_21} and \ref{tbl:PSU432_22} show that at least $126$ subgroups are needed for $cl_{10}$.  Assume that $m_3$ subgroups from $\calM_3$ and $m_4$ subgroups from $\calM_4$ are used in the cover; thus $m_3 + m_4 \ge 126$.  Since
\[\frac{112 \cdot 4860 - 4320m_3}{4860} = 112 - \frac{8m_3}{9}, \hspace{1cm} \frac{112 \cdot 3240 - 2880m_4}{3240} = 112 - \frac{8m_4}{9},\]
at least $112 - 8m_3/9$ subgroups are still needed to cover elements from $cl_{12,1}$ and at least $112 - 8m_4/9$ subgroups are still needed to cover elements from $cl_{18}$. Because 
\[\left(112 - \frac{8m_3}{9} \right) + \left(112 - \frac{8m_4}{9} \right) = 224 - \frac{8}{9} \cdot (m_3 + m_4),\]
at least $224 - 8(m_3+m_4)/9$ subgroups are needed to cover the remaining elements from these two classes.  Since
\[(m_3 + m_4) +\left( 224 - \frac{8}{9} \cdot (m_3 + m_4)\right) = 224 + \frac{1}{9}(m_3 + m_4) \ge 224 + \frac{1}{9} \cdot 126 = 238,\]
at least $238$ subgroups are needed to cover classes $cl_{10}$, $cl_{12,1}$, and $cl_{18}$, collectively.  Nothing from $cl_7$ has yet been covered, so the subgroup in $\calM_1$ is still needed.  This gives the lower bound of $239$. 

(viii) $\PSU(4,3).2$ is the group $\U(4,3).2_3$ in the ATLAS \cite{ATLAS}. Using $\GAP$, we note the following distribution of elements in $\PSU(4,3).2$.
 
 \begin{center}
\begin{table}[H]\centering
\begin{tabular}{c|c|c|c|c|c|c}
 { }        & $\calM_1$ & $\calM_2$ & $\calM_3$ & $\calM_4$  & $\calM_5$ & $\calM_6$ \\ 
 { }        &  (1)      & (112)   & (162)   & (162)  & (280) & (540)\\
 \hline
 $cl_{24}$ & 0 & 0 & 0 & 0 & $972,P$ & 0   \\
 $cl_{10}$ & 0 & 0 & $8064_2$ & 0 & 0 & 0  \\
$cl_{8}$ & 0 & $14580_2$ &  0 & 10080 & 0 & 0  \\

 \end{tabular}
 \caption{Element distribution in $\PSU(4,3).2$}
 \label{tbl:PSU432_31}
 \end{table}
 \end{center}
 
 \begin{center}
\begin{table}[H]\centering
\begin{tabular}{c|c|c|c|c|c}
 { }        & $\calM_7$ & $\calM_8$ & $\calM_9$ & $\calM_{10}$  & $\calM_{11}$ \\ 
 { }        &  (2835)      & (4536)   & (4536)   & (45366)  & (8505)\\
 \hline
 $cl_{24}$ & 96 & 0 & 0 & 0 & 0   \\
 $cl_{10}$ & 0 & 144 & 144 & 144 & 0  \\
$cl_{8}$ & 288 & 0 &  0 & 360 & 96  \\

 \end{tabular}
 \caption{Element distribution in $\PSU(4,3).2$, cont.}
 \label{tbl:PSU432_32}
 \end{table}
 \end{center}
Using $\GAP$, we see that the subgroups in $\calM_2$, $\calM_3$, and $\calM_5$ constitute a cover, giving the upper bound.  On the other hand, Tables \ref{tbl:PSU432_31} and \ref{tbl:PSU432_32} show that it takes at least $280$ subgroups to cover the elements in $cl_{24}$.  No maximal subgroup that contains elements from $cl_{24}$ contains elements of $cl_{10}$, so it takes at least $162/2$ subgroups to cover these elements.  Finally, because
\[ \left\lceil \frac{112\cdot \frac{14580}{2} - 288 \cdot 280}{14580} \right\rceil = 51,\]
it takes at least an additional $51$ subgroups to cover the elements of $cl_8$.  Hence it takes at least $412$ groups to cover these three classes.  The result follows.

(ix) The upper bound comes from using Algorithm \ref{alg}. On the other hand, using $\GAP$, we have the following distribution of elements in $\PSL(4,3).2$.
 
 \begin{center}
\begin{table}[H]\centering
\begin{tabular}{c|c|c|c|c|c|c|c|c|c}
 { }        & $\calM_1$ & $\calM_2$ & $\calM_3$ & $\calM_4$  & $\calM_5$ & $\calM_6$ & $\calM_7$ & $\calM_8$ & $\calM_9$\\ 
 { }        &  (1)      & (117)   & (117)   & (130)  & (520) & (1080) & (2106) & (8424) & (10530)\\
 \hline
 $cl_{12,1}$ & 0 & 0 & $4320,P$ & $3888,P$ & 0 & 0 & $240$ & 0 & 48\\
 $cl_{12,2}$ & 0 & $4320,P$ & 0 & $3888,P$ & 0 & 0 & $240$ & 0 & 48\\
 $cl_{6}$ & 0 & 0 & 0 & $10368_2$ & $1296,P$ & 1872 & 0 & 240 & 64\\
 $cl_{20}$ & $606528,P$ & 0 & 0 & 0 & 0 & 0 & 288 & 0 & 0\\
 $cl_{8}$ & 0 & $6480,P$ & $6480,P$ & 0 & $2916_2$ & 1404 & 0 & 180 & 0\\
 $cl_{10,1}$ & 0 & $10368_2$ & 0 & 0 & 0 & 0 & 288 & 0 & 0\\
 $cl_{10,2}$ & 0 & 0 & $10368_2$ & 0 & 0 & 0 & 288 & 0 & 0\\
 \end{tabular}
 \caption{Element distribution in $\PSL(4,3).2$}
 \label{tbl:PSL432}
 \end{table}
 \end{center}
Assume that $117 - m_{2,3}$ subgroups are used from classes $\calM_2$ and $\calM_3$ to cover $cl_8$.  In this case, an additional $m$ subgroups from classes $\calM_5$, $\calM_6$, and $\calM_8$ are needed to cover $cl_8$.  Now, $6480m_{2,3} \le 2916m$, and so $117 - m_{2,3} + m \ge 117 + 11m/20$.  At this point, we have covered at most $1872m$ elements of $cl_6$, and so there are $673920 - 1872m$ elements still to cover.  Since
\[\frac{673920 - 1872m}{10368} = 65 - \frac{13m}{72},\]
at least $65 - 13m/72$ subgroups are needed to cover the remaining elements of $cl_6$.  So far, we have used at least $182$ subgroups, since
\[\left(117 + \frac{11m}{20}\right) + \left(65 - \frac{13m}{72}\right) = 182 + \frac{133m}{360} \ge 182.\]
None of the subgroups that contain elements from $cl_6$ or $cl_8$ contain elements from $cl_{20}$, so including the subgroup from $\calM_1$ means at least $183$ subgroups are needed to cover $cl_6$, $cl_8$, and $cl_{20}$.  Of the elements in $cl_{12,1}$ and $cl_{12,2}$, we have covered at most $4320\cdot 117 + 3888 \cdot 65$, which leaves at least $252720$ still uncovered.  This means at least an additional $\lceil 252720/4320 \rceil$ more subgroups are needed, and, since $\lceil 252720/4320 \rceil = 59$, we have a lower bound of $242$ subgroups.

(x) The upper bound comes from Algorithm \ref{alg}.  Using $\GAP$, we have the following distribution of elements in ${\rm O}^-(8,2)$.
 
 \begin{center}
\begin{table}[H]\centering
\begin{tabular}{c|c|c|c|c|c|c|c|c}
 { }        & $\calM_1$ & $\calM_2$ & $\calM_3$ & $\calM_4$  & $\calM_5$ & $\calM_6$ & $\calM_7$ & $\calM_8$ \\ 
 { }        &  (119)      & (136)   & (765)   & (1071)  & (1632) & (24192) & (45696) & (1175040) \\
 \hline
 $cl_{17}$ & 0 & 0 & 0 & 0 & 0 & $480,P$ & 0 & 0 \\
 $cl_{30}$ & 0 & 0 & 0 & $6144,P$ & 0 & 0 & 144 & 0   \\
$cl_{21}$ & 0 & 0 &  $24576_2$ & 0 & $5760$ & 0 & 0 & 0   \\
 $cl_{9}$ & $368640_2$ & $161280,P$ & 0 & 0 & 0 & 0 & 0 & 0 \\
$cl_{15}$ & 0 & $96768_3$ & 0 & 0 & 5376 & 0 & 96 & 0 \\
 \end{tabular}
 \caption{Element distribution in ${\rm O}^-(8,2)$}
 \label{tbl:O-82}
 \end{table}
 \end{center}
It is clear from Table \ref{tbl:O-82} that at least $24192 + 1071$ subgroups are needed to cover $cl_{17}$ and $cl_{30}$.  No maximal subgroups that contain elements in $cl_{17}$ or $cl_{30}$ contain elements in $cl_{21}$, so at least another $\lceil 765/2 \rceil$ are needed.  Finally, no subgroup that contains elements in $cl_{17}$, $cl_{30}$, or $cl_{21}$ contains elements in $cl_9$, which takes at least $\lceil 119/2 \rceil$ additional subgroups, giving a lower bound of $25706$. 
 
(xi) Algorithm \ref{alg} shows that the covering number of ${\rm O}^+(8,2)$ is at most $765$.  Using $\GAP$, we have the following element distribution.
 
 \begin{center}
\begin{table}[H]\centering
\begin{tabular}{c|c|c|c|c|c|c|c|c}
 { }        & $\calM_1$ & $\calM_2$ & $\calM_3$ & $\calM_4$  & $\calM_5$ & $\calM_6$ & $\calM_7$ & $\calM_8$ \\ 
 { }        &  (120)      & (120)   & (120)   & (135)  & (135) & (135) & (960) & (960) \\
 \hline
 $cl_{15,1}$ & 0 & 0 & $96768,P$ & $172032_2$ & 0 & 0 & 0 & 0 \\
 $cl_{15,2}$ & $96768,P$ & 0 & 0 & 0 & 0 & $172032_2$ & 24192 & 0   \\
$cl_{15,3}$ & 0 & $96768,P$ &  0 & 0 & $172032_2$ & 0 & 0 & 24192   \\

 \end{tabular}
 \caption{Element distribution in ${\rm O}^+(8,2)$}
 \label{tbl:O+82_1}
 \end{table}
 \end{center}
 
  \begin{center}
\begin{table}[H]\centering
\begin{tabular}{c|c|c|c|c|c|c|c|c|c}
 { }        & $\calM_9$ & $\calM_{10}$ & $\calM_{11}$ & $\calM_{12}$  & $\calM_{13}$ & $\calM_{14}$ & $\calM_{15}$ & $\calM_{16}$ & $\calM_{17}$\\ 
 { }        &  (960)      & (1120)   & (1120)   & (1120)  & (1575) & (11200) & (12096) & (12096)) & (12096))\\
 \hline
 $cl_{15,1}$ & 24192 & 0 & 10368 & 0 & 0 & 0 & 0 & 960 & 0 \\
 $cl_{15,2}$ & 0 & 0 & 0 & 10368 & 0 & 0 & 0 & 0 & 960   \\
$cl_{15,3}$  & 0 & 10368 & 0 & 0 & 0 & 0 & 960 & 0 & 0   \\

 \end{tabular}
 \caption{Element distribution in ${\rm O}^+(8,2)$, cont.}
 \label{tbl:O+82_2}
 \end{table}
 \end{center}
It is clear from Tables \ref{tbl:O+82_1} and \ref{tbl:O+82_2} that no maximal subgroup contains elements from more than one of the classes $cl_{15,1}$, $cl_{15,2}$, or $cl_{15,3}$.  A minimal cover for each class consists of at least $\lceil 135/2 \rceil$ subgroups, and so at least $204$ subgroups are needed in any cover.
 
(xii) The upper bound comes from Algorithm \ref{alg}.  Using $\GAP$, we have the following distribution of elements.
 
 \begin{center}
\begin{table}[H]\centering
\begin{tabular}{c|c|c|c|c|c|c|c}
 { }        & $\calM_1$ & $\calM_2$ & $\calM_3$ & $\calM_4$  & $\calM_5$ & $\calM_6$ & $\calM_7$ \\ 
 { }        &  (1)      & (660)   & (660)   & (121)  & (121) & (144) & (3025) \\
 \hline
 $cl_{12}$ & 0 & 0 & $220_2$ & 0 & 0 & 0 & 24  \\
 $cl_{22}$ & 0 & $60,P$ & 0 & 0 & 0 & $275,P$ & 0  \\
 $cl_{6,1}$ & 0 & $220,P$ &  0 & $1200_2$ & $1200_2$ & 0 & $24$ \\
 $cl_{6,2}$ & $220,P$ & 0 &  0 & 0 & 0 & 0 & $4$ \\
 \end{tabular}
 \caption{Element distribution in $\PSL(2,11) \Wr 2$}
 \label{tbl:PSL211wr2}
 \end{table}
 \end{center}
Examining Table \ref{tbl:PSL211wr2}, it is clear that at least $660/2 + 144$ subgroups are needed to cover $cl_{12}$ and $cl_{22}$, since these elements lie in disjoint classes of maximal subgroups.  Since 
\[ \left\lceil \frac{145200 - 144 \cdot 220}{1200} \right\rceil = 95,\]
at least $95$ more subgroups needed for $cl_{6,1}$.  Finally, not all elements from $cl_{6,2}$ are covered, and so the subgroup from $\calM_1$ is needed, giving the lower bound of $570$.
\end{proof}

\begin{prop}
\label{prop:sigmageq}
We have the following lower bounds for the indicated covering number values:
\begin{enumerate}[(i)]
\item $\sigma(\PSL(5,3)) \ge 393030144$;
\item $\sigma((A_{11} \times A_{11}).4) \ge 213444$;
\item $\sigma(\PSL(7,2)) \ge 184308203520$.
\end{enumerate}
\end{prop}

\begin{proof}
(i) A Sylow $11$-subgroup of $\PSL(5,3)$ has order $121$ and is cyclic.  Using GAP (and/or \cite[Tables 8.18-8.19]{BrayHoltRoneyDougal}), there are $8$ classes of maximal subgroups, and only one has order divisible by $121$ (and hence is the only maximal subgroup containing an element of order $121$). A maximal subgroup in this class is isomorphic to $121:5$, and the index of one of these groups in G is $393030144$. The result follows.

(ii) By Algorithm \ref{alg} (or, more accurately, one iteration of the loop in Algorithm \ref{alg}), there exists a class $cl_{60}$ of elements of order $60$ that are distributed as follows.

 \begin{center}
\begin{table}[H]\centering
\begin{tabular}{c|c|c|c|c|c|c|c}
 { }        & $\calM_1$ & $\calM_2$ & $\calM_3$ & $\calM_4$  & $\calM_5$ & $\calM_6$ & $\calM_7$ \\ 
 { }        &  (1)      & (213444)   & (108900)   & (27225)  & (3025) & (121) & (131681894400) \\
 \hline
 $cl_{60}$ & 0 & 124416000 & $0$ & 0 & 0 & 0 & 0  \\
 \end{tabular}
 \caption{Element distribution in $(A_{11} \times A_{11}).4$}
 \label{tbl:A112.4}
 \end{table}
 \end{center}
The elements in the class $cl_{60}$ are partitioned among the subgroups in $\calM_2$, so at least $213444$ subgroups are needed.

(iii) This follows immediately from considering the elements of order $2^7 - 1$ and the result of Kantor \cite{Kantor} that only field extension subgroups contain a Singer cycle.
\end{proof}


\end{document}